\documentclass[10pt, leqno, a4paper]{amsart}
\usepackage[english]{babel}
\usepackage{amsthm}
\usepackage{graphicx}
\graphicspath{{Figures/}}

\title[Initial Tukey Structure Below A Stable Ordered-Union]{Initial Tukey Structure Below A Stable Ordered-Union Ultrafilter}
\author{Tan Özalp}
\address{Department of Mathematics, University of Notre Dame, Notre Dame, IN 46556, USA}
\email{aozalp@nd.edu}
\usepackage[all]{xy}
\usepackage{xcolor}
\usepackage{hyperref}
\hypersetup{hidelinks}
\usepackage{amsmath, amscd, amsfonts, amssymb, amsthm, mathtools, enumitem, cleveref}
\makeatletter
\newcommand{\leqnomode}{\tagsleft@true}
\newcommand{\reqnomode}{\tagsleft@false}
\makeatother

\usepackage{enumitem}

\newlist{SubItemList}{itemize}{1}
\setlist[SubItemList]{label={$-$}}

\let\OldItem\item
\newcommand{\SubItemStart}[1]{%
    \let\item\SubItemEnd
    \begin{SubItemList}[resume]%
        \OldItem #1%
}
\newcommand{\SubItemMiddle}[1]{%
    \OldItem #1%
}
\newcommand{\SubItemEnd}[1]{%
    \end{SubItemList}%
    \let\item\OldItem
    \item #1%
}
\newcommand*{\SubItem}[1]{%
    \let\SubItem\SubItemMiddle%
    \SubItemStart{#1}%
}%

\newtheorem{theorem}{Theorem}[section]
\newtheorem{theorem1}[theorem]{Claim}

\theoremstyle{definition}
\newtheorem*{definition1}{Remark}
\newtheorem{lemma}[theorem]{Lemma}
\newtheorem{corollary}[theorem]{Corollary}
\theoremstyle{definition}
\newtheorem{definition}[theorem]{Definition}
\newtheorem{theorem2}[theorem]{Fact}

\theoremstyle{definition}
\newtheorem{definition2}[theorem]{Question}
\theoremstyle{definition}

\begin{document}

\begin{abstract}

Answering a question of Dobrinen and Todorcevic asked in \cite{dobtod}, we prove that below any stable ordered-union ultrafilter $\mathcal{U}$, there are exactly four nonprincipal Tukey classes: $[\mathcal{U}], [\mathcal{U}_{\operatorname{min}}], [\mathcal{U}_{\operatorname{max}}]$, and $[\mathcal{U}_{\operatorname{minmax}}]$. This parallels the classification of ultrafilters Rudin-Keisler below $\mathcal{U}$ by Blass in \cite{Blass}. A key step in the proof involves modifying the proof of a canonization theorem of Klein and Spinas \cite{KlSp} for Borel functions on $\mathrm{FIN}^{[\infty]}$ to obtain a simplified canonization theorem for fronts on $\mathrm{FIN}^{[\infty]}$, recovering Lefmann's (\cite{Lef}) canonization for fronts of finite uniformity rank as a special case. We use this to classify the Rudin-Keisler classes of all ultrafilters Tukey below $\mathcal{U}$, which is then applied to achieve the main result.

\end{abstract}

\maketitle

\section{Introduction}

Tukey introduced the notion of \textit{Tukey ordering} to study the theory of convergence in topology \cite{MR2515}. The study of the Tukey ordering of a particular class of partial orders, namely that of the class of \textit{ultrafilters}, started with Isbell's \cite{MR201316} and independently Juhász's \cite{MR216467} constructions of ultrafilters with maximum Tukey degree. This study was revived with Milovich \cite{MR1500094}, and continued with a detailed investigation by Dobrinen and Todorcevic in \cite{dobtod}. Further major developments include \cite{diltod}, \cite{r1}, \cite{r2}, \cite{BDR}, \cite{dobell}, \cite{dobellinfty}, \cite{DMT}, \cite{dilshe}, \cite{kuzrag18} \cite{chains}, \cite{dobcont}, \cite{benhamou2023tukeytypesfubiniproducts}, \cite{benhamou2023commutativitycofinaltypes}, \cite{benhamou2024cofinaltypesultrafiltersmeasurable}, \cite{benhamou2024diamondprinciplestukeytopultrafilters}, and the survey papers \cite{dobsurvey}, \cite{dense}, and \cite{kuzeljevic2024orderstructureppointultrafilters}.

Two types of the important questions asked about the structure of the Tukey types of ultrafilters are the following: Which partial orders can be embedded into the Tukey ordering of the ultrafilters? Given an ultrafilter $\mathcal{U}$, what is the structure of the Tukey/Rudin-Keisler types of ultrafilters below $\mathcal{U}$? In Section $6$ of \cite{dobtod}, Dobrinen and Todorcevic obtained partial results on the initial Tukey structure below a stable ordered-union ultrafilter, but left the question of exact classification open. We shall answer this question in this paper.

In order to discuss our results, we now provide some definitions. Ultrafilters considered throughout the paper are assumed to be on countable base sets. For ultrafilters $\mathcal{U}$ and $\mathcal{V}$, we say \textit{$\mathcal{V}$ is Tukey reducible to $\mathcal{U}$} (or \textit{$\mathcal{V}$ is Tukey below $\mathcal{U}$}), and write $\mathcal{V} \leq_T \mathcal{U}$, if there is a map $f : \mathcal{U} \to \mathcal{V}$ which sends every filter base for $\mathcal{U}$ to a filter base for $\mathcal{V}$. We say that \textit{$\mathcal{U}$ and $\mathcal{V}$ are Tukey equivalent}, and write $\mathcal{U} \equiv_T \mathcal{V}$, if both $\mathcal{U} \leq_T \mathcal{V}$ and $\mathcal{V} \leq _T \mathcal{U}$. We call the collection of ultrafilters Tukey equivalent to $\mathcal{U}$ the \textit{Tukey type of $\mathcal{U}$}, and denote it by $[\mathcal{U}]$. We call the collection of Tukey types Tukey below $\mathcal{U}$ the \textit{initial Tukey structure below $\mathcal{U}$}.

The first initial Tukey structure result was Todorcevic's \footnote{It is explained in the introduction of \cite{diltod} that the result is due to Todorcevic.} proof of the Tukey minimality of Ramsey ultrafilters in \cite{diltod}, which mirrors the result of Blass that Ramsey ultrafilters are Rudin-Keisler minimal (\cite{blassramseyminimal}). Later, the initial Tukey structures below ultrafilters forced by to Laflamme's partial orders (\cite{laf}), and the isomorphism classes inside these Tukey classes were classified in \cite{r1} and \cite{r2}, analogously to Laflamme's results for the Rudin-Keisler order. Further work includes the classification of the initial Tukey structure and the isomorphism classes inside the Tukey classes below ultrafilters forced by $\mathcal{P}(\omega^k)/{\mathrm{Fin}}^{\otimes k}$ for all $k \geq 2$ in \cite{dobell}, and below ultrafilters associated to topological Ramsey spaces constructed from Fraïssé classes in \cite{DMT}. See the survey \cite{dense} for the full results.

We let $\mathrm{FIN}$ denote the set of all finite nonempty subsets of $\omega$ and define the maps $\operatorname{sm} : \mathrm{FIN}^{[\infty]} \to \omega$, $\operatorname{id} :\mathrm{FIN} \to \mathrm{FIN}$, $\operatorname{min} : \mathrm{FIN} \to \omega$, $\operatorname{max} : \mathrm{FIN} \to \omega$ and $\operatorname{minmax} : \mathrm{FIN} \to \omega^2$ by $\operatorname{sm}(s) = \varnothing$, $\operatorname{id}(s) = s$, $\operatorname{min}(s) = \text{the least element of $s$}$, $\operatorname{max}(s) = \text{the greatest} \text{ element of $s$}$, and $\operatorname{minmax}(s) = (\operatorname{min}(s), \operatorname{max}(s))$ for all $s \in \mathrm{FIN}$. The object of our study is an ultrafilter $\mathcal{U}$ on the base set $\mathrm{FIN}$ satisfying a certain partition property, called a \textit{stable ordered-union ultrafilter} (see \Cref{sou}). Letting $\bf{1}$ denote a principal ultrafilter and writing $\mathcal{U}_{\operatorname{min}}$, $\mathcal{U}_{\operatorname{max}}$ and $\mathcal{U}_{\operatorname{minmax}}$ for the $\mathrm{RK}$-images of $\mathcal{U}$ under the respective maps (see \Cref{rk}), we have ${\bf 1}\leq_T \mathcal{U}_{\operatorname{min}}, \mathcal{U}_{\operatorname{max}} \leq_T \mathcal{U}_{\operatorname{minmax}} \leq_T \mathcal{U}$. 

We fix an arbitrary stable ordered-union ultrafilter $\mathcal{U}$ on the base set $\mathrm{FIN}$. Denote the Rudin-Keisler equivalence of ultrafilters by $\cong$ (\Cref{rk}), and let us note that Rudin-Keisler reduction implies Tukey reduction. Recall the following theorem of Blass:

\begin{theorem}[\cite{Blass}]\label{blass}

Let $\mathcal{V}$ be a nonprincipal ultrafilter on $\omega$ such that $\mathcal{V} \leq_{RK} \mathcal{U}$. Then $\mathcal{V}$ is Rudin-Keisler equivalent to exactly one of $\mathcal{U}, \mathcal{U}_{\operatorname{min}}, \mathcal{U}_{\operatorname{max}}$ or $\mathcal{U}_{\operatorname{minmax}}$. Moreover, $\mathcal{U}_{\operatorname{min}}$ and $\mathcal{U}_{\operatorname{max}}$ are $\mathrm{RK}$-incomparable. This results in the following picture, where the arrows represent strict Rudin-Keisler reducibility:

\begin{figure}[h]
$$\xymatrix{
&\mathcal{U} \ar[d] &\\
&\mathcal{U}_{\operatorname{minmax}} \ar[ld] \ar[rd] &\\
\mathcal{U}_{\operatorname{min}}\ar[rd] & & \mathcal{U}_{\operatorname{max}}\ar[ld]\\
&{\mathbf 1}}
$$
\caption{Rudin-Keisler classes below $\mathcal{U}$.}\label{Fig1}
\end{figure}
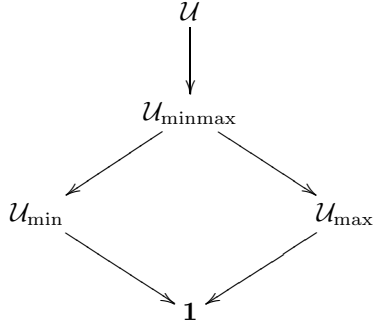

\end{theorem}

As for the Tukey order, Corollary $69$ and the subsequent remark in \cite{dobtod} showed that $\mathcal{U}_{\operatorname{min}}$ and $\mathcal{U}_{\operatorname{max}}$ are Tukey-incomparable. Theorem $74$ of the same paper included the construction of a particular stable ordered-union $\mathcal{U}$ such that $\mathcal{U} >_T \mathcal{U}_{\operatorname{minmax}}$, assuming $\mathrm{CH}$. This results in a similar picture to \Cref{Fig1} with respect to the Tukey order, but the exact Tukey structure below $\mathcal{U}$ was not classified, even under $\mathrm{CH}$.

\Cref{question} and \Cref{question2} below are Question $77$ and Question $75$ in \cite{dobtod}. \Cref{question3} is motivated by prior work on initial Tukey structures:

\begin{definition2}\label{question}

What ultrafilters are Tukey reducible to $\mathcal{U}$?

\end{definition2}

\begin{definition2} \label{question3}

What are the isomorphism types of ultrafilters Tukey reducible to $\mathcal{U}$?

\end{definition2}

\begin{definition2}\label{question2}

If $\mathcal{U}$ is any stable ordered-union, does it follow that $\mathcal{U} >_T \mathcal{U}_{\operatorname{minmax}}$?

\end{definition2}

Let $\mathcal{U}$ be an arbitrary stable ordered-union ultrafilter. We answer \Cref{question} in \Cref{main2} by proving that every nonprincipal ultrafilter $\mathcal{V} \leq_{T} \mathcal{U}$ is Tukey equivalent to exactly one of $\mathcal{U}, \mathcal{U}_{\operatorname{min}}, \mathcal{U}_{\operatorname{max}}$ or $\mathcal{U}_{\operatorname{minmax}}$, and so the Tukey structure below $\mathcal{U}$ is exactly like the Rudin-Keisler structure illustrated in \Cref{Fig1} (the exactness part of the statement is proved in \Cref{3.5}, which answers \Cref{question2}). \Cref{question3} is answered in \Cref{main1}, which proves that every nonprincipal ultrafilter Tukey below $\mathcal{U}$ is isomorphic to a countable Fubini iterate of ultrafilters from the set $\{\mathcal{U}, \mathcal{U}_{\operatorname{min}}, \mathcal{U}_{\operatorname{max}}, \mathcal{U}_{\operatorname{minmax}}\}$ (see \Cref{fubini}).

Recall the first initial Tukey structure result:

\begin{theorem}[\cite{diltod}]\label{minimality}

Let $\mathcal{U}$ be a Ramsey ultrafilter on $\omega$ and assume that $\mathcal{V}$ is nonprincipal with $\mathcal{V} \leq_T \mathcal{U}$. Then $\mathcal{V}$ is a countable Fubini iterate of $\mathcal{U}$, and in particular $\mathcal{V} \equiv_T \mathcal{U}$.

\end{theorem}

The proof of this theorem included Pudlák and Rödl's canonization theorem (\cite{pudrod}) as a key step. In \cite{r1} and \cite{r2}, new topological Ramsey spaces dense in Laflamme's forcings were constructed, and new canonization results for equivalence relations (equivalently, functions) on the fronts on these spaces were proved to be employed in the classifications of the initial Tukey structures. Ramsey-classification results, i.e., canonization theorems for fronts on the newly constructed topological Ramsey spaces, were proved and used as a key step of the classifications of the initial Tukey structures in \cite{dobell}, \cite{dobellinfty}, and \cite{DMT}, as well.

In this paper, we will work with Milliken's (\cite{milliken}) topological Ramsey space $\mathrm{FIN}^{[\infty]}$ (see \Cref{block}), and adapt the main result of \cite{KlSp} to prove a simplified canonization result for equivalence relations on fronts on $\mathrm{FIN}^{[\infty]}$. This will yield Lefmann's canonization in \cite{Lef} as a special case. The canonization theorem will then be used in the proof of the main result.

\Cref{2} provides the notation and states some well-known facts that will be used throughout the paper. \Cref{3} states more facts related to $\mathcal{U}$. In that section, we prove that $\mathcal{U}$ is rapid and $\mathcal{U} >_T \mathcal{U}_{\operatorname{minmax}}$, in $\mathrm{ZFC}$. This answers \Cref{question2}, which was asked in \cite{dobtod}, and was restated in \cite{kuzeljevic2024orderstructureppointultrafilters}. In \Cref{4}, we modify the main theorem of \cite{KlSp} to prove a canonization theorem that will better serve our purposes. Finally, \Cref{5} includes the main results, \Cref{main1} and \Cref{main2}.

\subsection*{Acknowledgments}

I am profoundly grateful to my advisor Natasha Dobrinen for introducing me to this problem, for her constant support and guidance, and for her invaluable feedback throughout the writing process. I would also like to thank Burak Kaya for his influence on my education and for his helpful comments.

\section{Notation and Background}\label{2}

Whenever $f$ is a function and $X \subseteq \operatorname{dom}(f)$, we will use $f''X$ and $f[X]$ interchangeably to mean the image of $X$ under $f$. Let us define the following well-known notions:

\begin{definition}\label{cofinal}

Let $\mathcal{U}$ and $\mathcal{V}$ be ultrafilters on countable base sets $I_1$ and $I_2$, respectively.

\begin{enumerate}[label=(\roman*)]
    \item $\mathcal{B} \subseteq \mathcal{U}$ is a \textit{filter base} for (or a \textit{cofinal subset} of) $\mathcal{U}$ if for every $U \in \mathcal{U}$, there is $B \in \mathcal{B}$ such that $B \subseteq U$.
    \item $f : \mathcal{U} \to \mathcal{V}$ is \textit{monotone} if $A \subseteq B \Rightarrow f(A) \subseteq f(B)$ for all $A, B \in \mathcal{U}$.
    \item $f : \mathcal{U} \to \mathcal{V}$ is \textit{cofinal} if for every base $\mathcal{B} \subseteq \mathcal{U}$, $f''\mathcal{B}$ is a base for $\mathcal{V}$.
\end{enumerate}

\end{definition}

Note that if $f : \mathcal{U} \to \mathcal{V}$ is a monotone map, then $f$ is cofinal if and only if $f''\mathcal{U}$ is a filter base for $\mathcal{V}$.

Let $I$ be a countable base set. Let $\mathcal{B}$ be a set of infinite subsets of $I$. We define $\langle \mathcal{B} \rangle = \{X \subseteq I : (\exists Y_1, \ldots, Y_n \in \mathcal{B}) \ Y_1 \cap \ldots \cap Y_n \subseteq X\}$.

Let us also define the appropriate notion of isomorphism between ultrafilters:

\begin{definition}\label{rk}

Let $\mathcal{U}$ and $\mathcal{V}$ be ultrafilters on countable base sets $I_1$ and $I_2$, respectively. Let $f: I_1 \to I_2$ be a map. Then $f(\mathcal{U}) = \{Y \subseteq I_2 : (\exists X \in \mathcal{U}) \ f[X] \subseteq Y\} = \{Y \subseteq I_2 : f^{-1}(Y) \in \mathcal{U}\}$ is an ultrafilter on $I_2$ called the \textit{$\mathrm{RK}$-image of $\mathcal{U}$ under $f$}. We say \textit{$\mathcal{V}$ is Rudin-Keisler} \textit{reducible to $\mathcal{U}$} (or \textit{$\mathcal{V}$ is Rudin-Keisler} \textit{below $\mathcal{U}$}) and write $\mathcal{V} \leq_{RK} \mathcal{U}$, if there is a map $f : I_1 \to I_2$ such that $f(\mathcal{U}) = \mathcal{V}$. We say that $\mathcal{U}$ and $\mathcal{V}$ are \textit{isomorphic} or \textit{Rudin-Keisler equivalent} and write $\mathcal{U} \cong \mathcal{V}$, if there are maps $f_1 : I_1 \to I_2$ and $f_2 : I_2 \to I_1$ such that $f_1(\mathcal{U}_1) = \mathcal{V}$ and $f_2(\mathcal{V}) = \mathcal{U}$. Equivalently, $\mathcal{U} \cong \mathcal{V}$ if there is a bijection (equivalently, an injection) $f : I_1 \to I_2$ such that $f(\mathcal{U}) = \mathcal{V}$. More information on the Rudin-Keisler order on ultrafilters can be found in \cite{blasssurvey} and \cite{halbeisen}.

\end{definition}

It follows from the definitions that $\mathcal{V} \leq_{RK} \mathcal{U} \Rightarrow \mathcal{V} \leq_{T} \mathcal{U}$. The following fact can be found in \cite{dobtod}:

\begin{theorem2} \label{2.3}

For ultrafilters $\mathcal{U}$ and $\mathcal{V}$, if $\mathcal{U} \geq_T \mathcal{V}$, then there exists a monotone cofinal map $f : \mathcal{U} \to \mathcal{V}$.

\end{theorem2}

For an ultrafilter $\mathcal{U}$ on the base set $I$ and $U \in \mathcal{U}$, let $\mathcal{U} \upharpoonright U$ denote $\{V \in \mathcal{U} : V \subseteq U\}$; note that $\mathcal{U} \upharpoonright U$ is an ultrafilter on the base set $U$. Consider $i : U \to I$ given by $i(x) = x$ for all $x \in U$. Then $i$ is injective and for all $X \in \mathcal{U} \upharpoonright U$, we get $i''X = X \in \mathcal{U}$. Hence,

\begin{theorem2}

For any $U \in \mathcal{U}$, $\mathcal{U} \cong \mathcal{U} \upharpoonright U$.

\end{theorem2}

\begin{definition}\label{block}

By a \emph{block sequence}, we mean a sequence $X : \operatorname{dom}(X)$ $ \to$ $\mathrm{FIN}$ such that $\operatorname{max}(X(i)) < \operatorname{min}(X(i+1))$ for all $i \in \operatorname{dom}(X) \in \omega \cup \{\omega\}$. $\sqsubseteq$ will mean initial segment and $\sqsubset$ will mean proper initial segment. We assume that the empty sequence is an initial segment of every sequence. Let $\mathrm{FIN}^{[n]}$ denote the set of all finite block sequences of length $n \in \omega$, and let ${\mathrm{FIN}}^{[\infty]}$ denote the set of all infinite block sequences. For $n \in \omega \cup \{\omega\}$, we denote $\mathrm{FIN}^{[< n]} = \bigcup_{k < n} \mathrm{FIN}^{[k]}$, including the empty sequence. Finally, when we refer to a topology on $\mathrm{FIN}^{[\infty]}$, we refer to the topology arising from the metric $d(X, Y) = \frac{1}{2^i}$ for $X \neq Y \in \mathrm{FIN}^{[\infty]}$, where $i \in \omega$ is the least with $X(i) \neq Y(i)$.

\end{definition}

\begin{definition}

Throughout the paper capital letters $A, B, C, X, Y$, and $Z$ denote elements of $\mathrm{FIN}^{[\infty]}$; small letters $a, b, c, x, y,$ and $z$ denote elements of $\mathrm{FIN}^{[< \infty]}$; $s,t $ and $u$ denote elements of $\mathrm{FIN}$.

\begin{enumerate}[label=(\roman*)]
    \item $s <_b t$ if and only if $\operatorname{max}(s) < \operatorname{min}(t)$. Similarly, $s <_b a$ if and only if $\operatorname{max}(s) < \operatorname{min}(a(0))$ and $a <_b b$ if and only if $\operatorname{max}(a(|a|-1)) < \operatorname{min}(b(0))$, whenever $a,b \neq \varnothing$. $s <_b X$ and $a <_b X$ are defined similarly.
    \item Assume that $s, a <_b b$. $s^{\smallfrown}b$ denotes the block sequence $(s, b(0),$ $ b(1), \ldots,$ $b(|b|-1))$. $a^{\smallfrown}b$ is defined similarly.
    \item $[X] = \{s \in \mathrm{FIN} : s = \bigcup_{i \in I} X(i) \ \text{for some finite $I \subseteq \omega$}\}$ and $[a] = \{s \in \mathrm{FIN} : s = \bigcup_{i \in I} a(i) \ \text{for some finite $I \subseteq \omega$}\}$. Whenever we write expressions of the kind ``$[X] \in \mathcal{U}$", $X$ is implicitly assumed to be an infinite block sequence.
    \item $X \leq Y$ if and only if for all $i \in \omega$, $X(i) \in [Y]$. $a \leq X$ and $a \leq b$ are defined similarly. $[X]^{[\infty]}$ denotes the set of infinite block sequences $X'$ with $X' \leq X$. For $n \in \omega$, $[X]^{[n]}$ denotes the set of block sequences $a \in \mathrm{FIN}^{[n]}$ with $a \leq X$, and $[X]^{[<n]} = \bigcup_{n' < n} [X]^{[n']}$.
    \item Let $n \in \omega$. $X / n = (X(i))_{i \geq i_0}$, where $i_0$ is the least index with $\operatorname{min}(X(i_0)) > n$. We will also write $X / s$ to mean $X / \operatorname{max}(s)$ for $s \in \mathrm{FIN}$, and $X / a$ to mean $X / \operatorname{max}(a(|a|-1))$ for $\varnothing \neq a \in \mathrm{FIN}^{[<\infty]}$.
    \item $X \leq^* Y$ if and only if there is $n \in \omega$ with $X / n \leq Y$.
\end{enumerate}

\end{definition}

For $n \in \omega \setminus \{0\}$ and $X \in \mathrm{FIN}^{[\infty]}$, let $r_n(X) = (X(0), X(1), \ldots, X(n-1))$. $r_0(X)$ is defined to be the empty sequence. Recall the famous theorem of Hindman:

\begin{theorem}[\cite{hind}]\label{hindmanorg}

Let $r \in \omega$. For every coloring $c : \mathrm{FIN} \to r$, there is some $X \in \mathrm{FIN}^{[\infty]}$ such that $c \upharpoonright [X]$ is constant.

\end{theorem}

It was first proved by Milliken in \cite{milliken} that $(\mathrm{FIN}^{[\infty]}, \leq, r)$ is a topological Ramsey space. The following is a straightforward corollary of \Cref{hindmanorg}, proved in \cite{Tay}:

\begin{lemma}[\cite{Tay}]\label{hindman}

Let $1\leq l < \omega$, $X \in \mathrm{FIN}^{[\infty]}$, and $r \in \omega$. For any coloring $c : [X]^{[l]} \to r$, there is $Y \leq X$ such that $c \upharpoonright [Y]^{[l]}$ is constant.

\end{lemma}

Finally, let us briefly review the theory of fronts on $\mathrm{FIN}^{[\infty]}$, which we take from \cite{argy}.

\begin{definition}[\cite{argy}]\label{front1}

Let $\mathcal{B} \subseteq \mathrm{FIN}^{[<\infty]}$.
\begin{enumerate}[label=(\roman*)]
    \item $\mathcal{B}$ is \textit{Ramsey} if for every $\mathcal{X} \subseteq \mathcal{B}$ and for every $X \in \mathrm{FIN}^{[\infty]}$, there is $Y \leq X$ such that $\mathcal{B} | Y \subseteq \mathcal{X}$ or $\mathcal{B} | Y \cap \mathcal{X} = \varnothing$, where $\mathcal{B} | Y = \{a \in \mathcal{B} : a \leq Y\}$. 
    \item $\mathcal{B}$ is \textit{Nash-Williams} if for all $a \neq b \in \mathcal{B}$, $a \not\sqsubset b$.
    \item $\mathcal{B}$ is a \textit{front on $X \in \mathrm{FIN}^{[\infty]}$} if $\mathcal{B}$ is Nash-Williams, $\mathcal{B} \subseteq [X]^{[<\infty]}$, and for all $Y \leq X$, there is a (unique) $a \in \mathcal{B}$ with $a \sqsubset Y$.
\end{enumerate}

\end{definition}

For $\mathcal{B} \subseteq \mathrm{FIN}^{[<\infty]}$ and $s \in \mathrm{FIN}$, we let $\mathcal{B}_{(s)} = \{s <_b a \in [\mathrm{FIN}]^{[<\infty]} : s^{\smallfrown} a \in \mathcal{B}\}$. Similarly, for $a \in \mathrm{FIN}^{[<\infty]}$, we will write $\mathcal{F}_{(a)} = \{a <_b b \in \mathrm{FIN}^{[<\infty]} : a^{\smallfrown} b \in \mathcal{F}\}$.

\begin{definition}[\cite{argy}]

A set $\mathcal{A} \subseteq \mathrm{FIN}$ is called \textit{small} if it does not contain $[X]$ for any $X \in \mathrm{FIN}^{[\infty]}$. Let $X \in \mathrm{FIN}^{[\infty]}$, let $\mathcal{B} \subseteq \mathrm{FIN}^{[<\infty]}$ be a family, and let $\alpha < \omega_1$. $\mathcal{B}$ is \textit{$\alpha$-uniform} on $X$ if one of the following holds: 

\begin{enumerate}[label=(\roman*)]
    \item $\alpha = 0$ and $\mathcal{B} = \{\varnothing\}$,
    \item $\alpha = \beta + 1$ and for all $s \in [X]$, $\mathcal{B}_{(s)}$ is $\beta$-uniform on $[X/s]$.
    \item $\alpha$ is a limit ordinal and for all $s \in [X]$, $\mathcal{B}_{(s)}$ is $\alpha_s$-uniform for some $\alpha_s < \alpha$, and for each $\gamma < \alpha$ $\{s \in [X] : \alpha_s \leq \gamma\}$ is small.
\end{enumerate}
We call $\mathcal{B}$ a \emph{uniform} family on $X$ if it is $\alpha$-uniform for some $\alpha < \omega_1$.

\end{definition}

Note that for every $X \in \mathrm{FIN}^{[\infty]}$, every uniform family on $X$ is a front. Also, it follows by induction on $\alpha$ that if $\mathcal{F}$ is $\alpha$-uniform on $X$ and $Y \leq X$, then $\mathcal{F} |Y$ is also $\alpha$-uniform on $Y$. In this case, we will call $\alpha$ the \emph{uniformity rank} of $\mathcal{F}$. Observe that fronts of finite uniformity rank have the form $[X]^{[n]}$ for $X \in \mathrm{FIN}^{[\infty]}$, and $n \in \omega$.

\begin{theorem2}[\cite{milliken}]\label{nashwilliams}

Every Nash-Williams family is Ramsey.

\end{theorem2}

For $s,t \in \mathrm{FIN}$, we define $s \trianglelefteq t$ iff $\operatorname{max}(s \Delta t) \in t$. Then $\trianglelefteq$ well-orders $\mathrm{FIN}$ with order-type $\omega$. For $a,b \in \mathrm{FIN}^{[<\infty]}$, we let $a <_{lex} b$ iff $a \sqsubset b$ or $a(i) \vartriangleleft b(i)$, where $i$ is minimal with $a(i) \neq b(i)$. If $\mathcal{F}$ is a front on some $X$, then $<_{lex}$ well-orders $\mathcal{F}$, and the following can be proved by induction on the order-type of $\mathcal{F}$ with respect to this ordering:

\begin{lemma}[\cite{argy}]\label{uniform}

Let $\mathcal{F}$ be a front on $X \in \mathrm{FIN}^{[\infty]}$. Then there is $Y \leq X$ such that $\mathcal{F}|Y$ is an $\alpha$-uniform front on $Y$, for some $\alpha \in \omega_1$.

\end{lemma}

The defining property of the ultrafilters we are interested in is that they give witnesses to \Cref{hindman}. These ultrafilters are called \emph{stable ordered-union} ultrafilters, which were studied in detail by Blass and Hindman in \cite{Blass} and \cite{BlHi}.

\begin{definition} \label{sou}

Let $\mathcal{U}$ be a an ultrafilter on the base set $\mathrm{FIN}$. $\mathcal{U}$ is called \textit{ordered-union}, if there is a filter basis $\mathcal{B}$ for $\mathcal{U}$ such that every member of $\mathcal{B}$ is of the form $[X]$ for some $X \in \mathrm{FIN}^{[\infty]}$.
$\mathcal{U}$ is called \textit{stable ordered-union}, if it is ordered-union and for every $(X_i)_{i \in \omega} \subseteq \mathrm{FIN}^{[\infty]}$ with $[X_i] \in \mathcal{U}$ and $X_{i+1} \leq^* X_i$ for all $i \in \omega$, there is $[X] \in \mathcal{U}$ such that $X \leq^* X_i$ for all $i \in \omega$.

\end{definition}

In the following, \ref{NWp} is motivated by \Cref{nashwilliams} and \ref{tayp} is motivated by \Cref{taycanon}. The equivalence of \ref{prop1}, \ref{Ramsey}, \ref{infty}, and \ref{tayp} was proved in \cite{Blass}, and their equivalence to \ref{sel}, and \ref{NWp} is in \cite{mij} and \cite{yuanyuanzheng}.

\begin{theorem}\label{prop}

Let $\mathcal{U}$ be an ordered-union ultrafilter on $\mathrm{FIN}$. The following are equivalent:

\begin{enumerate}[label=\textnormal{(\roman*)}]
    \item $\mathcal{U}$ is a stable ordered-union ultrafilter. \label{prop1}
    \item For all $1 \leq n < \omega$ and for all $f : \mathrm{FIN}^{[n]} \to 2$, there is $[X] \in \mathcal{U}$ such that $f$ is constant on $[X]^{[n]}$ (Ramsey property). \label{Ramsey}
    \item If $\mathcal{A} \subseteq \mathrm{FIN}^{[\infty]}$ is metrically analytic, then there is $[X] \in \mathcal{U}$ such that $[X]^{[\infty]} \subseteq \mathcal{A}$ or $[X]^{[\infty]}  \cap \mathcal{A} = \varnothing$ ($\infty$-dimensional Ramsey property) \label{infty}. 
    \item If $[X_a] \in \mathcal{U}$ for all $a \in \mathrm{FIN}^{[<\infty]}$, then there is $[X] \in \mathcal{U}$ such that $[X / a] \subseteq [X_a]$ for all $a \leq X$ (selectivity). \label{sel}
    \item For every Nash-Williams family $\mathcal{F} \subseteq \mathrm{FIN}$ and every $\mathcal{A} \subseteq \mathcal{F}$, there is $[X] \in \mathcal{U}$ such that $\mathcal{F} | X \subseteq \mathcal{A}$ or $\mathcal{F} | X \cap \mathcal{A} = \varnothing$ (Nash-Williams property). \label{NWp}
    \item For every function $f :\mathrm{FIN} \to \omega$, there is $[X] \in \mathcal{U}$ and $c \in \{\operatorname{sm}, \operatorname{min}, \operatorname{max},$ $\operatorname{minmax}, \operatorname{id}\}$ such that; $f(s) = f(t)$ if and only if $c(s)=c(t)$, for all $s,t \in [X]$ (canonical partition property). \label{tayp}
\end{enumerate}

\end{theorem}

Lastly, we define the \emph{Fubini limit} of ultrafilters:

\begin{definition}\label{fubini}

Let $\mathcal{V}$ be an ultrafilter on $I$ and let $\mathcal{W}_i$ be ultrafilters on $A$ for all $i \in I$. We define the following ultrafilter on $I \times A$: $$\lim_{i \to \mathcal{V}} \mathcal{W}_i = \{X \subseteq I \times A : \{i \in I : \{a \in A : (i, a) \in X\} \in \mathcal{W}_i\} \in \mathcal{V}\}.$$ If all of the $\mathcal{W}_i$'s are the same ultrafilter $\mathcal{W}$, we denote $\mathcal{V} \cdot \mathcal{W} = \lim_{i \to \mathcal{V}} \mathcal{W}_i$. Finally, we will write $\mathcal{V}^2$ instead of $\mathcal{V} \cdot \mathcal{V}$.

\end{definition}

To finish this section, let us remark that one can construct a stable ordered-union ultrafilter $\mathcal{U}$ under $\mathrm{CH}$ (or just $\mathfrak{p} = \mathfrak{c}$) inductively, using \Cref{hindmanorg}. Another approach to construct $\mathcal{U}$ is the following: Take a generic filter $G$ for the poset $(\mathrm{FIN}^{[\infty]}, \leq^*)$ and let $\mathcal{U} = \langle \{[X] : X \in G\} \rangle$. However, the existence of a stable ordered-union is unprovable in $\mathrm{ZFC}$, since it implies the existence of selective ultrafilters (see \Cref{3.2}), and there are models of $\mathrm{ZFC}$ without selective ultrafilters (\cite{kunen}, \cite{nshelah}, \cite{she}). Finally, we mention that Raghavan and Steprāns recently solved a longstanding open problem, asked in \cite{Blass}, by constructing a model of $\mathrm{ZFC}$ with at least two non-isomorphic selective ultrafilters, but without any stable ordered-union ultrafilters \cite{raghavan2023stableorderedunionversusselective}.

\section{Properties of \texorpdfstring{$\mathcal{U}$}{Lg}}\label{3}

In this section, we recall some facts about $\mathcal{U}$ and answer a question asked in \cite{dobtod} by proving that $\mathcal{U} >_T \mathcal{U}_{\operatorname{minmax}}$ (\Cref{3.5}). Let us recall the following types of ultrafilters on $\omega$:

\begin{definition}

Let $\mathcal{U}$ be an ultrafilter on $\omega$.

\begin{enumerate}[label=(\roman*)]
    \item $\mathcal{U}$ is a \textit{p-point} if for every $\{X_n : n \in \omega\} \subseteq \mathcal{U}$, there is some $X \in \mathcal{U}$ such that $|X \setminus X_n| < \aleph_0$ for all $n \in \omega$.
    \item $\mathcal{U}$ is a \textit{q-point} if for each partition of $\omega$ into finite pieces $\{I_n : n \in \omega\}$, there is $X \in \mathcal{U}$ such that $|X \cap I_n| \leq 1$ for all $n \in \omega$.
    \item $\mathcal{U}$ is \textit{Ramsey} if for every $k, r \in \omega$ and every $c : [\omega]^k \to r$, there is $X \in \mathcal{U}$ such that $c \upharpoonright [X]^k$ is constant.
    \item $\mathcal{U}$ is \textit{rapid} if for all $f : \omega \to \omega$, there exists $X \in \mathcal{U}$ such that $|X \cap f(n)| \leq n$ for every $n \in \omega$.
\end{enumerate}

\end{definition}

In Theorem $3$ of \cite{Mil}, Miller proved the following alternate characterization of rapid ultrafilters:

\begin{definition}

Let $\mathcal{V}$ be a nonprincipal ultrafilter on the countable base set $I$. $\mathcal{V}$ is called \textit{rapid} if one of the following equivalent conditions holds:

\begin{enumerate}[label=(\roman*)]
    \item Given finite subsets $P_n \subseteq I$, there is $X \in \mathcal{V}$ with $|X \cap P_n| \leq n$ for each $n \in \omega$.
    \item There is $h : \omega \to \omega$ such that given finite subsets $P_n \subseteq I$, there is $X \in \mathcal{V}$ with $|X \cap P_n| \leq h(n)$ for each $n \in \omega$.
\end{enumerate}

\end{definition}

Now let $\mathcal{U}$ be a stable ordered-union ultrafilter. Denote the $\mathrm{RK}$-images by $\mathcal{U}_{\operatorname{min}} = \operatorname{min}(\mathcal{U})$, $\mathcal{U}_{\operatorname{max}} = \operatorname{max}(\mathcal{U})$, and $\mathcal{U}_{\operatorname{minmax}} = \operatorname{minmax}(\mathcal{U})$, which are ultrafilters on $\omega$, $\omega$ and $\omega^2$, respectively. By definition, $\mathcal{U}_{\operatorname{min}}, \mathcal{U}_{\operatorname{max}} \leq_{RK} \mathcal{U}_{\operatorname{minmax}} \leq_{RK} \mathcal{U}$; hence also $\mathcal{U}_{\operatorname{min}}, \mathcal{U}_{\operatorname{max}} \leq_{T} \mathcal{U}_{\operatorname{minmax}} \leq_{T} \mathcal{U}$. In the following theorem \ref{3.2i}-\ref{3.2iii} are well-known. We prove \ref{3.2iv} because it seems to be missing from the previous literature.

\begin{theorem}\label{3.2}

\begin{enumerate}[label=\textnormal{(\roman*)}]
    \item {\rm(}\cite{BlHi}{\rm)} $\mathcal{U}_{\operatorname{min}}$ and $\mathcal{U}_{\operatorname{max}}$ are selective. \label{3.2i}
    \item {\rm(}Lemma $3.7$, \cite{Blass}{\rm)} $\mathcal{U}_{\operatorname{min}} \cdot \mathcal{U}_{\operatorname{max}} \cong \mathcal{U}_{\operatorname{minmax}}$. \label{3.2iii}
    \item {\rm(}Fact $65$, \cite{dobtod}{\rm)} $\mathcal{U}_{\operatorname{minmax}}$ is rapid, but not a p-point or a q-point. \label{3.2ii}
    \item {\rm(}Fact $65$, \cite{dobtod}{\rm)} $\mathcal{U}$ is not a p-point or a q-point.
    \item $\mathcal{U}$ is rapid. \label{3.2iv}
\end{enumerate}

\end{theorem}

\begin{proof}

We prove \ref{3.2iv}: Let $P_n \subseteq \mathrm{FIN}$ be finite for all $n\in \omega$. For each $n \in \omega$, let $B_n \in \mathrm{FIN}^{[n]}$ denote the block sequence $(\{0\}, \ldots, \{n-1\}).$ Define $h : \omega \to \omega$ by $h(n) = |[B_n]|$ for all $n \in \omega$. Let $$\mathcal{X} = \{X \in \mathrm{FIN}^{[\infty]} : (\forall n \in \omega) \ |[X] \cap P_n| \leq h(n)\}.$$ Note that $\mathcal{X}$ is closed. Observe that for every $X \in \mathrm{FIN}^{[\infty]}$, there is $X' \leq X$ with $X' \in \mathcal{X}$. Thus, by \Cref{prop} part \ref{infty}, there is $[Z] \in \mathcal{U}$ with $Z \in \mathcal{X}$, and we get the result. \qedhere

\end{proof}

To set the context for the Tukey order below $\mathcal{U}$, let us quote the following theorems:

\begin{theorem} \label{tukey}

\begin{enumerate}[label=\textnormal{(\roman*)}]
    \item $\mathcal{U}_{\operatorname{min}}$ and $\mathcal{U}_{\operatorname{max}}$ are nonisomorphic \cite{Blass} and Tukey incomparable \cite{dobtod}.
    \item {\rm(}Lemma $34$, \cite{dobtod}{\rm)} $\mathcal{U}_{\operatorname{min}} \cdot \mathcal{U}_{\operatorname{max}} \equiv_T \mathcal{U}_{\operatorname{max}} \cdot \mathcal{U}_{\operatorname{min}}$.
    \item {\rm(}Theorem $35$, \cite{dobtod}{\rm)} $\mathcal{U}_{\operatorname{min}}^2 \equiv_T \mathcal{U}_{\operatorname{min}}$ and $\mathcal{U}_{\operatorname{max}}^2 \equiv_T \mathcal{U}_{\operatorname{max}}$.
    \item {\rm(}\cite{benhamou2023tukeytypesfubiniproducts}{\rm)} $\mathcal{U}_{\operatorname{minmax}}^2 \equiv_T \mathcal{U}_{\operatorname{minmax}}$.
    \item {\rm(}\cite{benhamou2023tukeytypesfubiniproducts}{\rm)} $\mathcal{U}^2 \equiv_T \mathcal{U}$.
\end{enumerate}

\end{theorem}

To finish this section, we show that $\mathcal{U} >_T \mathcal{U}_{\operatorname{minmax}}$. Following the notation in \cite{dobcont}, for a subset $A \subseteq \omega^2$, we let $\hat{A}=\{b : (\exists a \in A) \ b \sqsubseteq a\}$. For $A \subseteq [\omega]^{<\omega}$ and $n \in \omega$, we let $A \upharpoonright n = \{a \in A : \operatorname{max}(a) < n\}$. Since $\mathcal{U}_{\operatorname{minmax}} \cong \mathcal{U}_{\operatorname{min}} \cdot \mathcal{U}_{\operatorname{max}}$ and $\mathcal{U}_{\operatorname{min}} $ and $\mathcal{U}_{\operatorname{max}}$ are p-points, Theorem $4.4$ of \cite{dobcont} applied to $\mathcal{U}_{\operatorname{min}}$ and $\mathcal{U}_{\operatorname{max}}$ in our notation yields:

\begin{lemma} \label{3.4}

Let $\mathcal{V}$ be an ultrafilter on the base set $I$ and let $f : \mathcal{U}_{\operatorname{minmax}} \to \mathcal{V}$ be a monotone cofinal map. Then there is $[\tilde{X}] \in \mathcal{U}$, an increasing sequence $(n_k)_{k \in \omega} \subseteq \omega$, and a map $\hat{f} : \{\hat{A} \upharpoonright n : A = [X]_{\operatorname{minmax}} \upharpoonright n \ \text{for some} \ X \leq \tilde{X} \ \text{and} \ n \in \omega\} \to \mathcal{P}(I)$ such that for all $X \leq \tilde{X}$ with $[X] \in \mathcal{U}$, we have $f([X]_{\operatorname{minmax}}) = \bigcup_{k \in \omega} \hat{f}(\hat{A_k})$, where $A_k = [X]_{\operatorname{minmax}} \upharpoonright n_k$ for each $k \in \omega$.

\end{lemma}

In this case, we say that the finitary map $\hat{f}$ \textit{generates} $f$ on $\{\hat{A} \upharpoonright n : A = [X]_{\operatorname{minmax}} \upharpoonright n \ \text{for some} \ X \leq \tilde{X} \ \text{with} \ [X] \in \mathcal{U}, \ \text{and} \ n \in \omega\}$. Recall that $\mathcal{U} \geq_T \mathcal{U}_{\operatorname{minmax}}$. We now answer Question $75$ of \cite{dobtod} using this lemma:

\begin{theorem}\label{3.5}

$\mathcal{U} >_T \mathcal{U}_{\operatorname{minmax}}$.

\end{theorem}

\begin{proof}

Assume towards a contradiction that $\mathcal{U}_{\operatorname{minmax}} \geq_T \mathcal{U}$. Then there is a monotone cofinal $f : \mathcal{U}_{\operatorname{minmax}} \to \mathcal{U}$. By \Cref{3.4}, we may take the corresponding $[\tilde{X}] \in \mathcal{U}$ and the finitary map $\hat{f}$ which generates $f$ on $\{\hat{A} \upharpoonright n : A = [X]_{\operatorname{minmax}} \upharpoonright n \ \text{for some} \ X \leq \tilde{X} \ \text{with} \ [X] \in \mathcal{U}, \ \text{and} \ n \in \omega\}$. Note that $\hat{f}$ generates a map, say $g$, on the set $\{X \in \mathrm{FIN}^{[\infty]} : X \leq \tilde{X}\}$ by $g([X]_{\operatorname{minmax}}) = \bigcup_{k \in \omega} \hat{f}(\hat{A_k})$, where $A_k = \hat{[X]}_{\operatorname{minmax}} \upharpoonright n_k$ for each $k \in \omega$, and $g$ agrees with $f$ on $\{X \in \mathrm{FIN}^{[\infty]} : X \leq \tilde{X} \ \text{and} \ [X] \in \mathcal{U}\}$.

Let us note that $\{(X, Y) \in \mathrm{FIN}^{[\infty]} \times \mathrm{FIN}^{[\infty]} : X \leq Y\}$ is a closed subset of $\mathrm{FIN}^{[\infty]} \times \mathrm{FIN}^{[\infty]}$. Moreover, for fixed $n_0 \in \omega$ and $s_0 \in \mathrm{FIN}$, $\{(X,Y) : s_0 \in \hat{f}(\hat{A}), \ \text{where} \ A = [X]_{\operatorname{minmax}} \upharpoonright n_{k_0} \ \text{and} \ s_0 \notin [Y]\}$ is closed. We conclude that $\{(X, Y) \in \mathrm{FIN}^{[\infty]} \times \mathrm{FIN}^{[\infty]} : X \leq \tilde{X}, f([X]_{\operatorname{minmax}}) \subseteq [Y]\}$ is also a closed subset of $\mathrm{FIN}^{[\infty]} \times \mathrm{FIN}^{[\infty]}$.

Now define $\mathcal{A}_1 = \{X \in \mathrm{FIN}^{[\infty]} : [X] \cap [\tilde{X}] = \varnothing\}$, $\mathcal{A}_2 = \{X \leq \tilde{X} : (\exists X' \leq X) \ g([X']_{\operatorname{minmax}}) \subseteq [X]\}$, and $\mathcal{A}_3 = \{X \leq \tilde{X} : g([X]_{\operatorname{minmax}}) \cap [X] = \varnothing\}$. By the previous paragraph, we see that $\mathcal{A}_1$ is closed, $\mathcal{A}_2$ is analytic, and $\mathcal{A}_3$ is closed.

We let $\mathcal{A}^c = \mathcal{A}_1^c \cap \mathcal{A}_2 \cap \mathcal{A}_3^c$. It follows that $\mathcal{A}^c$ is analytic, so \Cref{prop} part \ref{infty} is applicable to $\mathcal{A}^c$. Hence, there is $[Z] \in \mathcal{U}$ such that $[Z]^{[\infty]} \subseteq \mathcal{A}$ or $[Z]^{[\infty]} \cap \mathcal{A} = \varnothing$. Applying the argument from the proof of Theorem $74$ in \cite{dobtod} that was used to construct a stable ordered-union $\mathcal{U}$ with $\mathcal{U} >_T \mathcal{U}_{\operatorname{minmax}}$, we will show that the latter case is not possible.

Assume towards a contradiction that $[Z]^{[\infty]} \cap \mathcal{A} = \varnothing$. If there is no $X' \leq Z, \tilde{X}$, then we can use \Cref{prop} part \ref{Ramsey} to find $[X] \in \mathcal{U}$ with $[X] \cap [\tilde{X}] = \varnothing$. Hence we would have $X \in \mathcal{A}_1$, which contradicts $[Z]^{[\infty]} \cap \mathcal{A} = \varnothing$.

Now we can assume that there is some $X \leq Z, \tilde{X}$. If there is $W \leq X$ such that for all $W' \leq W$, we have $g([W']_{\operatorname{minmax}}) \nsubseteq [W]$, then $W \in \mathcal{A}_2^c$, which contradicts $[Z]^{[\infty]} \cap \mathcal{A} = \varnothing$ again.

Otherwise, for all $W \leq X$, there is some $W' \leq W$ with $g([W']_{\operatorname{minmax}})$ $ \subseteq [W]$. Let $W(i) = X(3i) \cup X(3i+1) \cup X(3i+2)$ for all $i \in \omega$. Find $W' \leq W$ such that $f([W']_{\operatorname{minmax}}) \subseteq [W]$. Then $W'(j) = \bigcup_{i \in I_j} W(i)$ for all $j \in \omega$, where $(I_j)_{j \in \omega}$ is itself a block sequence. Let $m_j = \operatorname{min}(I_j)$ and $k_j = \operatorname{max}(I_j)$. For all $j \in \omega$, set $Y(j) = X(3m_j) \cup X(3k_j+2)$. Then $[W']_{\operatorname{minmax}} = [Y]_{\operatorname{minmax}}$, $[W] \cap [Y] = \varnothing$ and $Y \leq X$. Finally, we have $$g([Y]_{\operatorname{minmax}}) = g([W']_{\operatorname{minmax}}) \subseteq [W],$$ which is disjoint from $[Y]$. Hence, $Y \in \mathcal{A}_3$ and we contradict $[Z]^{[\infty]} \cap \mathcal{A} = \varnothing$ in this case as well.

It follows that $[Z]^{[\infty]} \subseteq \mathcal{A}$, and in particular $Z \in \mathcal{A}$. If $Z \in \mathcal{A}_1$, then $[Z] \cap [\tilde{X}] = \varnothing$, which cannot happen. If $Z \in \mathcal{A}_2^c$, then there is no $[X'] \in \mathcal{U}$ with $g([X']_{\operatorname{minmax}}) \subseteq [Z]$, which contradicts the cofinality of $g$. Finally, if $Z \in \mathcal{A}_3$, then $g([Z]_{\operatorname{minmax}}) \cap [Z] = \varnothing$, which is again a contradiction. \qedhere

\end{proof}

\section{Canonization Theorems on \texorpdfstring{$\mathrm{FIN}^{[\infty]}$}{Lg}}\label{4}

This section adapts a canonization result of Klein and Spinas in \cite{KlSp} to be applied in the classification of the initial Tukey structure in the next section. The results to be stated are about defining a class of \emph{canonical} equivalence relations, and showing that every equivalence relation restricts to such a relation. Since every equivalence relation corresponds to a function and vice versa, we can state the results mentioning only functions.

We define the following maps as in \cite{KlSp}: For $s \in \mathrm{FIN}$; $\operatorname{sm}(s) = \varnothing$, $\operatorname{min-sep}(s) = \{$the least element of $s\}$, $\operatorname{max-sep}(s) = \{\text{the greatest element of $s$}\}$, $\operatorname{minmax-sep}(s)$ $= \operatorname{min-sep}(s) \cup \operatorname{max-sep}(s)$, and $\operatorname{sss}(s) = s = \operatorname{vss}(s)$.

To motivate our result in this section, let us recall the remarkable theorem of Taylor:

\begin{theorem}[\cite{Tay}]\label{taycanon}

For every function $f : \mathrm{FIN} \to \omega$, there are $X \in \mathrm{FIN}^{[\infty]}$ and $c \in \{\operatorname{sm}, \operatorname{min-sep}, \operatorname{max-sep}, \operatorname{minmax-sep}, \operatorname{id}\}$ such that for all $s, t \in [X]$; $$f(s) = f(t) \ \text{if and only if} \ c(s) = c(t).$$ 

\end{theorem}

This theorem was generalized by Klein and Spinas \cite{KlSp} in the following way:

\begin{definition}[\cite{KlSp}]

Let $\gamma : \mathrm{FIN}^{[< \infty]} \to \{\operatorname{sm}, \operatorname{min-sep}, \operatorname{max-sep},$ $\operatorname{minmax-sep},$ $\operatorname{sss}, \operatorname{vss}\}$. Let $X \in \mathrm{FIN}^{[\infty]}$. Let $k_0 = 0$ and let $(k_i : 0 < i < N \leq \omega)$ increasingly enumerate those $k$ such that $\gamma(X \upharpoonright (k-1)) = \operatorname{vss}$. Also let $k(N) = \omega$ if $N < \omega$. We define $$\Gamma_{\gamma}(X) = (\bigcup_{k_i \leq j < k_{i+1}} \gamma(X \upharpoonright j)(X(j)) : i < N),$$ with the convention that evaluating $\varnothing$ at some $s \in \mathrm{FIN}$ means we ignore $s$.

\end{definition}

\begin{theorem}[\cite{KlSp}]\label{klsp}

Let $\Delta : \mathrm{FIN}^{[\infty]} \to \mathbb R$ be a Borel map. Then there exists $A \in \mathrm{FIN}^{[\infty]}$ and $\gamma : \mathrm{FIN}^{[< \infty]} \to \{\operatorname{sm}, \operatorname{min-sep}, \operatorname{max-sep},$ $\operatorname{minmax-sep}, \operatorname{sss}, \operatorname{vss}\}$ such that for all $X, Y \leq A$; $$\Delta(X) = \Delta(Y) \ \text{if and only if} \ \Gamma_{\gamma}(X) = \Gamma_{\gamma}(Y).$$

\end{theorem}

In this section, using the same machinery as in \cite{KlSp} and the lemmas used in that paper, we will prove a simplified version of \Cref{klsp} for functions defined on fronts on $\mathrm{FIN}^{[\infty]}$, which will be utilized in \Cref{5} to find the exact Tukey structure below $\mathcal{U}$. This simplification includes the notion of \emph{admissibility} (\Cref{gamma}), which is motivated by the canonization of Lefmann in \cite{Lef} (see \Cref{lefcanon}). As a result, the simplified version will recover Lefmann's canonization theorem for functions on fronts of finite uniformity rank as a special case. The essential modification we need concerns the definition of $\Gamma_{\gamma}$:

\begin{definition} \label{gamma}

Let $\mathcal{F}$ be a front on $X$. Define $\hat{\mathcal{F}} = \{a \in \mathrm{FIN}^{[<\infty]} : a \sqsubseteq b \ \text{for some $b \in \mathcal{F}$}\}$, including the empty sequence. Consider a map $\gamma : \hat{\mathcal{F}}\setminus \mathcal{F} \to \{\operatorname{sm}, \operatorname{min-sep},$ $\operatorname{max-sep}, \operatorname{minmax-sep}, \operatorname{sss}, \operatorname{vss}\}$. We shall say that $\gamma$ is \textit{admissible} if the following hold:

\begin{enumerate}[label=(\roman*)]
    \item For $a \in \hat{\mathcal{F}}$, if $\gamma(a) = \operatorname{sss}$, then for all $a <_b x \in [X]^{[<\infty]}$ with $a^{\smallfrown} x \in \hat{\mathcal{F}}$, if there is a least $k$ with $\gamma(a^{\smallfrown} (x \upharpoonright k)) = \operatorname{vss}$, then for every $l < k$ we have $\gamma(a^{\smallfrown} (x \upharpoonright l)) \in \{\operatorname{sm}, \operatorname{sss}\}$.
    \item Let $a \in \hat{\mathcal{F}} \setminus \mathcal{F}$. Then either for all $a <_b x \leq X$ with $a^{\smallfrown} x \in \mathcal{F}$ and for all $k \leq |x|$, we have $\gamma(a^{\smallfrown} (x \upharpoonright k)) = \operatorname{sm}$, or for all $a <_b x \leq X$ with $a^{\smallfrown} x \in \mathcal{F}$, there is $k \leq |x|$ such that $\gamma(a^{\smallfrown} (x \upharpoonright k)) \neq \operatorname{sm}$.
    \item Let $a \in \hat{\mathcal{F}} \setminus \mathcal{F}$ be with $\gamma(a)=\operatorname{min-sep}$, and assume that for all $a <_b x \leq X$ with $a^{\smallfrown} x \in \mathcal{F}$, there is $1 \leq k \leq |x|$ such that $a^{\smallfrown}(x \upharpoonright k) \neq \operatorname{sm}$. Then either for all $a <_b x \leq X$ with $a^{\smallfrown} x \in \mathcal{F}$, we have $\gamma(a^{\smallfrown}(x \upharpoonright l)) = \operatorname{max-sep}$, where $1\leq l$ is the least with $\gamma(a^{\smallfrown}(x \upharpoonright l)) \neq \operatorname{sm}$; or for all $a <_b x \leq X$ with $a^{\smallfrown} x \in \mathcal{F}$, we have $\gamma(a^{\smallfrown}(x \upharpoonright l)) \neq \operatorname{max-sep}$, where $1 \leq l$ is the least with $\gamma(a^{\smallfrown}(x \upharpoonright l)) \neq \operatorname{sm}$.
\end{enumerate}

\noindent For admissible $\gamma$ and $a \in \hat{\mathcal{F}}$, define $(k_i)_{0 \leq i < N}$ as follows: $k_0 = 0$ and given $|a| > k_i$, if $\gamma(a \upharpoonright k_i) = \operatorname{sss}$, then $|a| > k_{i+1} > k_i$ is the least with $\gamma(a \upharpoonright (k_{i+1}-1)) = \operatorname{vss}$ if it exists, otherwise $k_{i+1} = |a|$. If $\gamma(a \upharpoonright k_i) = \operatorname{min-sep}$ and the least $k_i < l < |t|$ (if it exists) with $\gamma(a \upharpoonright l) \neq \operatorname{sm}$ satisfies $\gamma(a \upharpoonright l) = \operatorname{max-sep}$, then $k_{i+1} = l+1$. Finally, if $\gamma(a \upharpoonright k_i) \neq \operatorname{sss}$ and we are not in the previous case, then $k_{i+1} = k_i+1$. Let $N$ be the least natural number with $k_N = |a|$. Now define $$\Gamma_{\gamma}(a) = (\bigcup_{k_i \leq j < k_{i+1}} \gamma(a \upharpoonright j)(a(j)) : 0 \leq i < N).$$

\end{definition}

Now we recall Lefmann's canonization theorem for functions on fronts of finite uniformity rank, in our notation:

\begin{theorem} [\cite{Lef}] \label{lefcanon}

For every $k \in \omega \setminus \{0\}$ and every function $f : \mathrm{FIN}^{[k]} \to \omega$, there are $X \in \mathrm{FIN}$ and admissible $\gamma : [X]^{[<k]} \to \{\operatorname{sm}, \operatorname{min-sep}, \operatorname{max-sep}, $ {\rm minmax-sep}, $\operatorname{sss}, \operatorname{vss}\}$ such that $\gamma(b)=\gamma(b')$ whenever $b,b' \in [X]^{[<k]}$ are with $|b| = |b'|$, and for all $a, a' \in [X]^{[<k]}$; $$f(a)=f(a') \ \text{if and only if} \ \Gamma_{\gamma}(a) = \Gamma_{\gamma}(a').$$

\end{theorem}

We will use the separating and mixing technique which has its origins in \cite{PrVo} and was the main tool in \cite{KlSp}. Analogously to \cite{KlSp}, for $a=(a(0), \ldots, a(n-1)) <_b b=(b(0),\ldots,b(m-1)) \in \mathrm{FIN}^{[<\infty]}$, $a ^{\smallfrown} b$ denotes the block sequence $(a(0),\ldots, a(n-1), b(0), \ldots, b(m-1))$, and $a \uparrow^{\smallfrown} b$ denotes the block sequence $(a(0),\ldots, a(n-1) \cup b(0), \ldots, b(m-1))$. Finally, $a \square$ is a placeholder to mean either $a$ or $a \uparrow$.

The main theorem of this section is the following, a simplification of \Cref{klsp} that employs \Cref{gamma} to define $\Gamma_{\gamma}$, which has \Cref{lefcanon} as a special case:

\begin{theorem}\label{canonization}

Let $\mathcal{F}$ be a front on some $X \in \mathrm{FIN}^{[\infty]}$, and let $g : \mathcal{F} \to \omega$ be a function. Then there is $Y \leq X$ and $\gamma : (\hat{\mathcal{F}}\setminus \mathcal{F})|Y \to \{\operatorname{sm}, \operatorname{min-sep},$ $\operatorname{max-sep},$ $\operatorname{minmax-sep}, \operatorname{sss}, \operatorname{vss}\}$ such that for each $a, a' \in \mathcal{F}|Y$, $$g(a) = g(a') \ \text{if and only if} \ \Gamma_{\gamma}(a) = \Gamma_{\gamma}(a').$$

\end{theorem}

Let us now define the notions we need and do the preparation for the proof of \Cref{canonization}. Fix a front $\mathcal{F}$ on $X \in \mathrm{FIN}^{[\infty]}$ and let $g : \mathcal{F} \to \omega$ be a function. We may as well assume that $\mathcal{F} \neq \{\varnothing\}$, since otherwise the result will follow trivially. By \Cref{uniform}, we may assume that $\mathcal{F}$ is $\alpha$-uniform for some $\alpha < \omega_1$ to ease up the computations. Let $\varnothing \neq a,b \in \hat{\mathcal{F}}$ and for $Y \leq X$, denote $Y/(a,b) = Y/\operatorname{max}(a(|a|-1) \cup b(|b|-1))$.

\begin{definition}[\cite{KlSp}]

For $Y \leq X$ and $a, b \in \hat{\mathcal{F}}$, we say that \textit{$Y$ separates $a\square$ and $b\square$} if for every $x, y \leq Y / (a, b)$ such that $a\square^{\smallfrown}x, b\square^{\smallfrown} y \in \mathcal{F}$, $g(a\square^{\smallfrown}x) \neq g(b\square^{\smallfrown}y)$ (note that we allow the cases $x = \varnothing$ or $y = \varnothing$). We say that \textit{$Y$ mixes $a\square$ and $b\square$} if no $Z \leq Y$ separates $a\square$ and $b\square$. Finally, we say that \textit{$Y$ decides $a\square$ and $b \square$} if it either separates or mixes $a\square$ and $b\square$.

\end{definition}

The proof of the following general Ramsey theoretic fact is a straightforward fusion argument and in the lines of the proof of Lemma $4.6$ in \cite{r1}:

\begin{lemma}\label{fusion}

1) Let $P(\cdot,\cdot)$ be a property such that:

\begin{enumerate}[label=(\roman*)]
    \item For all $a \in \hat{\mathcal{F}}$ and $Z \leq Y \leq X$, $P(a\square, Y) \Rightarrow P(a\square, Z)$ (hereditary).
    \item For all $a \in \hat{\mathcal{F}}$ and $Y \leq X$, there is $Z \leq Y$ with $P(a\square, Z)$ (density).
\end{enumerate}

Then for all $Y \leq X$, there is $Y' \leq Y$ such that for all $a \in \hat{\mathcal{F}} | Y'$ and for all $Z \leq Y'$, $P(a\square, Z/a)$ holds.

2) Let $Q(\cdot,\cdot,\cdot)$ be a property such that:

\begin{enumerate}[label=(\roman*)]
    \item For all $a, b \in \hat{\mathcal{F}}$ and $Z \leq Y \leq X$, $Q(a\square, b\square, Y) \Rightarrow Q(a\square, b\square, Z)$ (hereditary).
    \item For all $a,b \in \hat{\mathcal{F}}$ and $Y \leq X$, there is $Z \leq Y$ with $Q(a\square, b\square, Z)$ (density).
\end{enumerate}

Then for all $Y \leq X$, there is $Y' \leq Y$ such that for all $a,b \in \hat{\mathcal{F}} | Y'$ and for all $Z \leq Y'$, $Q(a\square, b\square, Z/(a,b))$ holds.

\end{lemma}

It follows by definition that for every $a,b \in \hat{\mathcal{F}}$ and for every $Y \leq X$, there is $Z \leq Y$ which decides $a\square$ and $b\square$. Also, if $Y$ decides $a\square$ and $b\square$, then any $Z \leq Y$ decides $a\square$ and $b\square$ in the same way as $Y$ does. This observation together with \Cref{fusion} implies the following:

\begin{lemma}[\cite{KlSp}]\label{decide}

For all $Y \leq X$, there is $Y' \leq Y$ such that for all $a,b \in \hat{\mathcal{F}} | Y'$, $a\square$ and $b\square$ are decided by $Y'$.

\end{lemma}

We note that our separating-mixing notions are technically different from the ones used in \cite{KlSp}. The notions there were defined with respect to block sequences starting after $a$ and $b$, since the function to be canonized there was a Borel function defined on $\mathrm{FIN}^{[\infty]}$. Here, our theorem will be a simplified version only for fronts, so the notions are adjusted in a suitable way. Since any function $g : \mathcal{F} \to \omega$ induces a unique continuous function on $\mathrm{FIN}^{[\infty]}$, almost all of the lemmas in \cite{KlSp} will have straightforward translations to our setting and the proofs will essentially be the same. As a result, it should be noted here that most of the lemmas and definitions in this section will be direct translations of those that were used in \cite{KlSp}, and they will be included for the readability of the proof of the main result in this section.

The first of such lemmas is the transitivity of mixing, which directly follows from Lemma $2.3$ in \cite{KlSp}:

\begin{lemma}[\cite{KlSp}]

Let $a,b,c \in \hat{\mathcal{F}}$ and let $Y \leq X$. If $Y$ mixes $a\square$ and $b\square$ and $Y$ mixes $b\square$ and $c\square$, then $Y$ mixes $a\square$ and $c\square$.

\end{lemma}

Now that we know mixing is an equivalence relation, let us note the following direct consequence of \Cref{taycanon}:

\begin{lemma}[\cite{KlSp}]\label{longlemma}

For every $a \in \hat{\mathcal{F}} \setminus \mathcal{F}$, there is $Y \leq X$ which decides every $b\square$ and $c\square$ with $b,c \in \hat{\mathcal{F}}|Y$, and one of the following holds:

\begin{enumerate}[label=(\roman*)]
    \item For all $s,t \in [Y / a]$, $a^{\smallfrown}s$ and $a^{\smallfrown}t$ are mixed by $Y$.
    \item For all $s,t \in [Y / a]$, $a^{\smallfrown}s$ and $a^{\smallfrown}t$ are mixed by $Y$ if and only if $\operatorname{min-sep}(s) = \operatorname{min-sep}(t)$.
    \item For all $s,t \in [Y / a]$, $a^{\smallfrown}s$ and $s^{\smallfrown}t$ are mixed by $Y$ if and only if $\operatorname{max-sep}(s) = \operatorname{max-sep}(t)$.
    \item For all $s,t \in [Y / a]$, $a^{\smallfrown}s$ and $a^{\smallfrown}t$ are mixed by $Y$ if and only if $\operatorname{minmax-sep}(s)$ $=\operatorname{minmax-sep}(t)$.
    \item For all $s,t \in [Y / a]$, $a^{\smallfrown}s$ and $a^{\smallfrown}t$ are mixed by $Y$ if and only if $s=t$.
\end{enumerate}

\end{lemma}

\begin{proof}

First, by \Cref{decide}, we pick $X' \leq X$ which decides every $b\square$ and $c\square$ with $b,c \in \hat{\mathcal{F}}|X'$. Then we note that for every $a \in \hat{\mathcal{F}} \setminus \mathcal{F}$, since $\mathcal{F}_{(a)}$ is $\alpha$-uniform on $X/a$ for some $1 \leq \alpha < \omega_1$, it follows that $a^{\smallfrown}s \in \hat{\mathcal{F}}$ for all $s \in [X'/a]$. To finish, we apply \Cref{taycanon} to the function that corresponds to the equivalence relation on $[X'/a]$ defined by $s E t$ if and only if $a^{\smallfrown}s$ and $a^{\smallfrown}t$ are mixed by $X'$, for $s,t \in [X'/a]$, and get the corresponding $Y \leq X'$. \qedhere

\end{proof}

Let us remark here that, if $a \in \mathcal{F}$ and $s \in [X/a]$, then similarly to the argument in the proof of \Cref{longlemma}, $(a\uparrow)^{\smallfrown}s \in \hat{\mathcal{F}}$.

We repeat the following definition from \cite{KlSp}:

\begin{definition}[\cite{KlSp}]

Let $a \in \hat{\mathcal{F}}$ and let $Y\leq X$.

\begin{enumerate}[label=(\roman*)]
    \item We say $a\square$ is \textit{strongly mixed by $Y$} if $a\square^{\smallfrown}s$ and $a\square^{\smallfrown}t$ are mixed by $Y$ for all $s,t \in [Y/a]$.
    \item We say $a$ is \textit{min-separated by $Y$} if for all $s,t \in [Y/a]$, $a^{\smallfrown}s$ and $a^{\smallfrown}t$ are mixed by $Y$ if and only if $\operatorname{min-sep}(s)=\operatorname{min-sep}(t)$.
    \item We say $a\square$ is \textit{max-separated by $Y$} if for all $s,t \in [Y/a]$, $a\square^{\smallfrown}s$ and $a\square^{\smallfrown}t$ are mixed by $Y$ if and only if $\operatorname{max-sep}(s)=$max-sep$(t)$.
    \item We say $a$ is \textit{minmax-separated by $Y$} if for all $s,t \in [Y/a]$, $a^{\smallfrown}s$ and $a^{\smallfrown}t$ are mixed by $Y$ if and only if $\operatorname{minmax-sep}(s) = \operatorname{minmax-sep}(t)$.
    \item We say $a\square$ is \textit{strongly separated by $Y$} if for all $s,t \in [Y/a]$, $a\square^{\smallfrown}s$ and $a\square^{\smallfrown}t$ are mixed by $Y$ if and only if $s=t$.
    \item Finally, we say \textit{$a\square$ is separated in some sense by $Y$} if one of the last four bullet points holds for $a\square$. Also, $a\square$ is \textit{completely decided by $Y$} if it is either strongly mixed or separated in some sense by $Y$.
\end{enumerate}

\end{definition}

Let us remark here that, for $a \in \mathcal{F}$, although this definition still makes sense for $a\uparrow$, it doesn't make sense for $a$. Thus, we will just declare $a$ to be strongly mixed whenever $a \in \mathcal{F}$.

The following straightforward consequence, Lemma $2.7$ in \cite{KlSp}, is the reason we omitted the $\square$ in some of the alternatives:

\begin{lemma}[\cite{KlSp}] \label{4.11}

Let $a \in \hat{\mathcal{F}} \setminus \mathcal{F}$ and $Y \leq X$.

\begin{enumerate}[label=(\roman*)]
    \item Let $a$ be strongly mixed by $Y$. Then $a^{\smallfrown}s\uparrow$ is strongly mixed by $Y$ for every $s \in [Y/a]$.
    \item Let $a$ be min-separated by $Y$. Then $a^{\smallfrown}s\uparrow$ is strongly mixed by $Y$ for every $s \in [Y/a]$.
    \item Let $a$ be max-separated by $Y$. Then $a^{\smallfrown}s\uparrow$ is max-separated by $Y$ for every $s \in [Y/a]$.
    \item Let $a$ be minmax-separated by $Y$. Then $a^{\smallfrown}s\uparrow$ is max-separated by $Y$ for every $s \in [Y/a]$.
    \item Let $a$ be strongly separated by $Y$. Then $a^{\smallfrown}s\uparrow$ is strongly separated by $Y$ for every $s \in [Y/a]$.
\end{enumerate}

\end{lemma}

\begin{lemma}[\cite{KlSp}]\label{4.12}

Let $a \in \hat{\mathcal{F}} \setminus \mathcal{F}$ and $Y \leq X$.

\begin{enumerate}[label={(\roman*)}]
    \item Let $a\square$ be strongly mixed by $Y$. Then $a\square$ and $a\square^{\smallfrown}s\square$ and also $a\square^{\smallfrown}s\square$ and $a\square^{\smallfrown}t\square$ are mixed by $Y$ for all $s,t \in [Y/a]$. \label{4.12i}
    \item Let $a$ be min-separated by $Y$. Then $a^{\smallfrown}s\square$ and $a^{\smallfrown}t\square$ are mixed by $Y$ for all $s,t \in [Y/a]$ with $\operatorname{min-sep}(s) = \operatorname{min-sep}(t)$. \label{4.12ii}
    \item Let $a\square$ be max-separated by $Y$. Then $a\square$ and $a\square^{\smallfrown}s\uparrow$ and also $a\square^{\smallfrown}s\uparrow$ and $a\square^{\smallfrown}t\uparrow$ are mixed by $Y$ for all $s,t \in [Y/a]$. \label{4.12iii}
    \item Let $a$ be minmax-separated by $Y$. Then $a^{\smallfrown}s\uparrow$ and $a^{\smallfrown}t\uparrow$ are mixed by $Y$ for all $s,t \in [Y/a]$ with $\operatorname{min-sep}(s) = \operatorname{min-sep}(t)$. \label{4.12iv}
\end{enumerate}

\end{lemma}

Now, \Cref{fusion} and \Cref{longlemma} together imply that:

\begin{lemma}[\cite{KlSp}]\label{4.13}

For all $Y \leq X$, there is $Y' \leq Y$ which completely decides every $a\square$ with $a \in \hat{\mathcal{F}}|Y'$.

\end{lemma}

We augment the definition of canonical in \cite{KlSp} by adding the last condition in the following:

\begin{definition}

We say that \textit{$Y \leq X$ is canonical for $g$} if the following hold:

\begin{enumerate}[label=(\roman*)]
    \item For all $a,b \in \hat{\mathcal{F}} | Y$, $Y$ decides $a\square$ and $b\square$. \label{canon1}
    \item Every $a \in \hat{\mathcal{F}}|Y$ is completely decided by $Y$. \label{canon2}
    \item Let $a,b \in (\hat{\mathcal{F}} \setminus \mathcal{F})|Y$. Then $a\square$ and $a\square^{\smallfrown}s\square$ are either separated by $Y$ for all $s \in [Y/a]$, or mixed by $Y$ for all $s \in [Y/a]$. If $a \in \mathcal{F}|Y$, then the same result holds for $a\uparrow$ and $(a\uparrow)^{\smallfrown}s\square$. Similarly, $a\square^{\smallfrown}s\square$ and $a\square^{\smallfrown}s\square$, $a\square^{\smallfrown}s\square$ and $b\square^{\smallfrown}s\square$, $a\square^{\smallfrown}s\square$ and $a\square^{\smallfrown}s\square^{\smallfrown}t\square$, and finally $a\square^{\smallfrown}s\square$ and $t\square^{\smallfrown}s\square^{\smallfrown}t\square$ are in each case either separated by $Y$ for all $s <_b t \in [Y/(a,b)]$ or mixed by $Y$ for all $s <_b t \in [Y/(a,b)]$, whenever there are $s_0 <_b t_0 \in [Y/(a,b)]$ with $a\square^{\smallfrown}s_0\square^{\smallfrown}t_0\square \in \hat{\mathcal{F}}|Y$ and $s_1 <_b t_1 \in [Y/(a,b)]$ with $b\square^{\smallfrown}s_1\square^{\smallfrown}t_1\square \in \hat{\mathcal{F}}|Y$. The analogous result also holds for $a\square=a\uparrow$ if $a \in \mathcal{F}|Y$. \label{canon3}
    \item If $a \in (\hat{\mathcal{F}} \setminus \mathcal{F}) | Y$, then either for all $x \leq Y$ such that $a\square^{\smallfrown}x \in \mathcal{F}$ and for all $1\leq k \leq |x|$, $a\square^{\smallfrown} (x \upharpoonright k)$ is strongly mixed by $Y$; or for all $x \leq Y$ with $a\square^{\smallfrown}x \in \mathcal{F}$, there is $1\leq k \leq |x|$ such that $a\square^{\smallfrown} (x \upharpoonright k)$ is separated in some sense by $Y$. The same also holds for $a\uparrow$ when $a \in \mathcal{F}|Y$. \label{canon4}
    \item Let $a \in \hat{\mathcal{F}} | Y$ and assume that $a$ is min-separated by $Y$. Assume also that for all $x \leq Y$ with $a^{\smallfrown}x \in \mathcal{F}$, there is $1\leq k \leq |x|$ such that $a^{\smallfrown} (x \upharpoonright k)$ is separated in some sense by $Y$. Then either for all $x \leq Y$ with $a^{\smallfrown}x \in \mathcal{F}$, the least $1\leq k \leq |x|$ such that $a^{\smallfrown} (x \upharpoonright k)$ is separated in some sense by $Y$ is max-separated by $Y$; or for all $x \leq Y$ with $a^{\smallfrown}x \in \mathcal{F}$, the least $1\leq k \leq |x|$ such that $a^{\smallfrown} (x \upharpoonright k)$ is separated in some sense by $Y$ is not max-separated by $Y$. \label{canon5}
\end{enumerate}

\end{definition}

\begin{lemma} \label{4.15}

There is $Y \leq X$ which is canonical for $g$.

\end{lemma}

\begin{proof} 

The proof proceeds along the lines of the proof of Lemma $2.10$ in \cite{KlSp}. By \Cref{decide} and \Cref{4.13}, we can find $A \leq X$ such that \ref{canon1} and \ref{canon2} of being canonical are satisfied for $A$. Next, we can apply \Cref{hindman} and \Cref{fusion} to $A$ to get $B\leq A$ such that \ref{canon1}-\ref{canon3} of being canonical are satisfied for $B$ in the following way:

For the case with $a$ and $a^{\smallfrown}s$, we color $[Z/a]$ for arbitrary $Z \leq A$, into two colors by $c(s) = 0$ if and only if $a$ and $a^{\smallfrown}s$ are mixed by $Z$. Then we pick homogeneous $Z_0 \leq Z$ for this coloring. Hence the density assumption of \Cref{fusion} is satisfied, which means that there is $B' \leq A$ for which \ref{canon3} of canonical is satisfied for the case of $a$ and $a^{\smallfrown}s$. The other mentioned combinations are handled via similar colorings. For example, assume that there are $s_0<_b t_0 \in [A/(a,b)]$ with $a^{\smallfrown}s_0^{\smallfrown}t_0 \in \hat{\mathcal{F}}|A$ and $s_1<t_1 \in [A/(a,b)]$ with $b^{\smallfrown}s_1^{\smallfrown}t_1 \in \hat{\mathcal{F}}|A$. Hence, $\mathcal{F}_{(a)}|A$ is $\alpha$-uniform for $\alpha$ at least $2$. The same assertion also holds for  $\mathcal{F}_{(b)}|A$. We first color $[A/(a,b)]^{[2]}$ into $2$ colors by $c(s,t) = 0$ if and only if $a^{\smallfrown}s^{\smallfrown}t \in \hat{\mathcal{F}}$. The homogeneous $Z \leq A$ then has to be homogeneous with color $0$, since $\alpha \geq 2$. Similarly, we get the corresponding $Z_0 \leq Z$ for $b$. Then for arbitrary $Z_1 \leq Z_0$, we define $c : [Z_1]^{[2]} \to 2$ by $c(s,t) = 0$ if and only if $a^{\smallfrown}s^{\smallfrown}t$ and $b^{\smallfrown}s^{\smallfrown}t$ are mixed by $Z_1$. We apply \Cref{hindman} to $c$ to get homogeneous $Z_2 \leq Z_1$. It follows that the density assumption of \Cref{fusion} is satisfied again.

Now we look at \ref{canon4}: For given $a \in (\hat{\mathcal{F}} \setminus \mathcal{F}) | B$ and $Z \leq B$, define $$\mathcal{X} = \{x \in \mathcal{F}_{(a\square)}|Z : \forall 1\leq k \leq |x| \ \text{$a\square^{\smallfrown} (x \upharpoonright k)$ is strongly mixed by $Z$}\}.$$ By \Cref{nashwilliams}, there is $Z' \leq Z$ such that $\mathcal{F}_{(a\square)}|Z' \subseteq \mathcal{X}$ or $\mathcal{F}_{(a\square)}|Z' \cap \mathcal{X} = \varnothing$. The argument is exactly the same for $a \uparrow$ when $a \in \mathcal{F}|B$. This shows that the density assumption holds, so it follows by \Cref{fusion} that there is $C \leq B$ such that \ref{canon4} holds for $C$.

\ref{canon5} is handled similarly to \ref{canon4} to get the final $Y \leq C$. \qedhere

\end{proof}

From now on we fix a canonical $Y_0 \leq X$ and let $Y' = (Y_0(3i) \cup Y_0(3i+1) \cup Y_0(3i+2) : i < \omega)$, which is also canonical, to make sure that the maps $\operatorname{min-sep}, \operatorname{max-sep},$ $\operatorname{minmax-sep}$, and $\operatorname{sss}$ will have disjoint images on $[Y']$. Finally, we let $Y = (Y'(3i) \cup Y'(3i+1) \cup Y'(3i+2) : i < \omega)$, for technical reasons. The rest of the results in this section until the main theorem hold for both $Y$ and $Y'$, but they will only be stated for $Y$.

Let $a \in \hat{\mathcal{F}}|Y$ be strongly separated by $Y$. Since $Y$ is canonical, as in \cite{KlSp}, we can split this into two cases:

\begin{enumerate}[label=(\roman*)]
    \item We say $a\square$ is \textit{still strongly separated by $Y$} if for all $s \in [Y/a]$, $a\square^{\smallfrown}s$ and $a^{\smallfrown}s\uparrow$ are mixed by $Y$.
    \item We say $a\square$ is \textit{very strongly separated by $Y$} if for all $s \in [Y/a]$, $a\square^{\smallfrown}s$ and $a^{\smallfrown}s\uparrow$ are separated by $Y$.
\end{enumerate}

It follows that for $a \in \hat{\mathcal{F}}|Y$, if $a\square$ is still strongly separated by $Y$, then $a\square^{\smallfrown}s \uparrow$ is still strongly separated by $Y$ for all $s \in [Y/a]$. The same holds for very strong separation as well.

We will need the technical lemmas used in \cite{KlSp}. The following includes Lemmas $2.14-2.23$ in \cite{KlSp} adapted to our setting. The proofs are exactly the same as the versions in \cite{KlSp}.

\begin{lemma}[\cite{KlSp}]\label{4.16}

Let $a,b \in \hat{\mathcal{F}}|Y$. In the following statement, all $x$ and $y$ are assumed to be nonempty.

\begin{enumerate}[label=(\roman*)]
    \item Assume that $a$ and $b$ are mixed by $Y$ and $a$ and $b$ are both min-separated by $Y$. If $x \in \mathcal{F}_{(a)}|Y$ and $y \in \mathcal{F}_{(b)}|Y$ are with $g(a^{\smallfrown}x) = g(b^{\smallfrown}y)$, then $\operatorname{min-sep}(x(0)) = \operatorname{min-sep}(y(0))$. In this case, $a^{\smallfrown}s$ and $b^{\smallfrown}t$ are mixed by $Y$ for all $s,t \in [Y/(a,b)]$ with $\operatorname{min-sep}(s)=\operatorname{min-sep}(t)$. \label{4.16i}
    \item Assume that $a\square$ and $b\square$ are mixed by $Y$ and $a\square$ and $b\square$ are both max-separated by $Y$. If $x \in \mathcal{F}_{(a\square)}|Y$ and $y \in \mathcal{F}_{(b\square)}|Y$ are with $g(a\square^{\smallfrown}x) = g(b\square^{\smallfrown}y)$, then $\operatorname{max-sep}(x(0)) = \operatorname{max-sep}(y(0))$. In this case, $a\square^{\smallfrown}s$ and $b\square^{\smallfrown}t$ are mixed by $Y$ for all $s,t \in [Y/(a,b)]$ with $\operatorname{max-sep}(s)=\operatorname{max-sep}(t)$. \label{4.16ii}
    \item Assume that $a$ and $b$ are mixed by $Y$ and $a$ and $b$ are both minmax-separated by $Y$. If $x \in \mathcal{F}_{(a)}|Y$ and $y \in \mathcal{F}_{(b)}|Y$ are with $g(a^{\smallfrown}x) = g(b^{\smallfrown}y)$, then $\operatorname{minmax-sep}(x(0)) = \operatorname{minmax-sep}(y(0))$. In this case, $a^{\smallfrown}s \uparrow$ and $b^{\smallfrown}s \uparrow$ are mixed by $Y$ for all $s \in [Y/(a,b)]$. Furthermore, $a^{\smallfrown}s$ and $b^{\smallfrown}t$ are mixed by $Y$ for all $s,t \in [Y/(a,b]$ with $\operatorname{min-sep}(s)=\operatorname{min-sep}(t)$ and $\operatorname{max-sep}(t) = \operatorname{max-sep}(t)$. \label{4.16iii}
    \item Assume that $a\square$ and $b\square$ are mixed by $Y$ and $a\square$ and $b\square$ are both strongly separated by $Y$. If $x \in \mathcal{F}_{(a\square)}|Y$ and $y \in \mathcal{F}_{(b\square)}|Y$ are with $g(a\square^{\smallfrown}x) = g(b\square^{\smallfrown}y)$, then either $x(0)$ is an initial segment of $y(0)$ or $y(0)$ is an initial segment of $x(0)$. \label{4.16iv}
    \item Assume that $a\square$ and $b\square$ are mixed by $Y$ and $a\square$ and $b\square$ are both very strongly separated by $Y$. If $x \in \mathcal{F}_{(a\square)}|Y$ and $y \in \mathcal{F}_{(b\square)}|Y$ are with $g(a\square^{\smallfrown}x) = g(b\square^{\smallfrown}y)$, then $x(0)=y(0)$. In this case, $a\square^{\smallfrown}s$ and $b\square^{\smallfrown}s$ are mixed by $Y$ for all $s \in [Y/(a,b)]$. \label{4.16v}

\end{enumerate}

\end{lemma}

Now we state the translations of Lemma $2.27$ and $2.28$ of \cite{KlSp}:

\begin{lemma}[\cite{KlSp}] \label{4.17}

Let $a,b \in \hat{\mathcal{F}}|Y$. Assume that $a$ and $b$ are mixed by $Y$. Furthermore, assume that $a$ is min-separated by $Y$ and $b$ is minmax-separated by $Y$. Then if $\varnothing \neq x \in \mathcal{F}_{(a\square)}|Y$ and $\varnothing \neq y \in \mathcal{F}_{(b\square)}|Y$ are with $g(a\square^{\smallfrown}x) = g(b\square^{\smallfrown}y)$, then $\operatorname{min-sep}(x(0)) = \operatorname{min-sep}(y(0))$ and $\operatorname{max-sep}(x(0)) <_b \operatorname{max-sep}(y(0))$. In this case, $a^{\smallfrown}s \uparrow$ and $b^{\smallfrown}s \uparrow$ are mixed by $Y$ for all $s \in [Y/(a,b)]$.

\end{lemma}

\begin{lemma}[\cite{KlSp}]\label{4.17new}

Let $a,b \in \hat{\mathcal{F}}|Y$. Assume that $a\square$ and $b\square$ are mixed by $Y$.

\begin{enumerate}[label=(\roman*)]
    \item Assume that both $a\square$ and $b\square$ are still strongly separated by $Y$. Then $a\square^{\smallfrown}s\square$ and $b\square^{\smallfrown}s\square$ are mixed by $Y$ for every $s \in [Y/(a,b)]$. \label{4.17newi}
    \item Assume that $a\square$ is still strongly separated by $Y$ and $b\square$ is very strongly separated by $Y$. Then $a\square^{\smallfrown}s\square$ and $b\square^{\smallfrown}s\uparrow$ are mixed by $Y$ for all $s \in [Y/(a,b)]$. Furthermore, $a\square^{\smallfrown}s\square$ and $b\square^{\smallfrown}s$ are separated by $Y$ for all $s \in [Y/(a,b)]$. \label{4.17newii}
    \item Assume that both $a\square$ and $b\square$ are very strongly separated by $Y$. Then $a\square^{\smallfrown}s\uparrow$ and $b\square^{\smallfrown}s\uparrow$ are mixed by $Y$ for all $s \in [Y/(a,b)]$. Furthermore, $a\square^{\smallfrown}s$ and $b\square^{\smallfrown}s \uparrow$ are separated by $Y$ for all $s \in [Y/(a,b)]$. \label{4.17newiii}
    \end{enumerate}

\end{lemma}

Finally, we state Lemmas $2.29-2.31$ of \cite{KlSp} :

\begin{lemma}[\cite{KlSp}] \label{4.18}

Let $a,b \in (\hat{\mathcal{F}}\setminus \mathcal{F})|Y$.

\begin{enumerate}[label=(\roman*)]
    \item Assume that $a$ and $b\square$ are mixed by $Y$. If $a$ is min-separated by $Y$, then $b\square$ is neither max-separated nor strongly separated by $Y$. The same result holds when $b \in \mathcal{F}|Y$ and $b\square = b\uparrow$ as well.
    \item Assume that $a\square$ and $b\square$ are mixed by $Y$. If $a\square$ is max-separated by $Y$, then $b\square$ is neither minmax-separated nor strongly separated by $Y$. The same result holds when $a,b \in \mathcal{F}|Y$ and $a\square = a\uparrow$ and $b\square=b\uparrow$ as well.
    \item Assume that $a$ and $b\square$ are mixed by $Y$. If $a$ is minmax-separated by $Y$, then $b\square$ is not strongly separated by $Y$. The same result holds when $b \in \mathcal{F}|Y$ and $b\square = b\uparrow$ as well.
\end{enumerate}

\end{lemma}

Now we start to prove new lemmas to be used in this paper.

\begin{lemma} \label{4.19}

Assume that $a \in \hat{\mathcal{F}}|Y$ is still strongly separated by $Y$. Let $a <_b x \leq Y$ be with $a^{\smallfrown} x \in \hat{\mathcal{F}}|Y$. If there is a least $k$ with $a^{\smallfrown} (x \upharpoonright k)$ very strongly separated by $Y$, then for any $l < k$ $a^{\smallfrown} (x \upharpoonright l)$ is either strongly mixed or still strongly separated by $Y$.
\end{lemma}

\begin{proof}

Let $x = (x(0), \ldots, x(|x|-1))$. Since $a$ is still strongly separated, we know that $a^{\smallfrown} x(0)$ and $a^{\smallfrown} x(0) \uparrow$ are mixed by $Y$, and we also know that $a^{\smallfrown} x(0) \uparrow$ is still strongly separated by $Y$. Now, \Cref{4.18} implies that $a^{\smallfrown} x(0)$ is either strongly mixed or strongly separated by $Y$. If $a^{\smallfrown}x(0)$ is very strongly separated by $Y$, then we are done. If $a^{\smallfrown}x(0)$ is still strongly separated by $Y$, then we can do the same argument to see that $a^{\smallfrown} x(0)^{\smallfrown} x(1)$ is either strongly mixed or strongly separated by $Y$. Finally, if $a^{\smallfrown}x(0)$ is strongly mixed by $Y$, then \Cref{4.12} part \ref{4.12i} implies that $a^{\smallfrown}x(0)$ and $a^{\smallfrown}x(0)^{\smallfrown} x(1)$ are mixed by $Y$, which means $a^{\smallfrown}x(0)\uparrow$ and $a^{\smallfrown}x(0)^{\smallfrown} x(1)$ are mixed by $Y$, and so $a^{\smallfrown}x(0)^{\smallfrown} x(1)$ is either strongly mixed or strongly separated by $Y$. Now fix $l < k$ and suppose that $a^{\smallfrown} (x \upharpoonright l')$ is either strongly mixed or strongly separated by $Y$ for all $l' < l$. Find the greatest $l_0 \leq l$ for which $a^{\smallfrown} (x \upharpoonright l_0)$ is strongly separated by $Y$. If $a^{\smallfrown} (x \upharpoonright l_0)$ is very strongly separated by $Y$, then we get the result. Otherwise, by the argument in the base case and repeated applications of \Cref{4.12} part \ref{4.12i}, $a^{\smallfrown} (x \upharpoonright l_0)$ and $a^{\smallfrown} (x \upharpoonright l)$ are mixed by $Y$, and $a^{\smallfrown} (x \upharpoonright l_0)$ is still strongly separated by $Y$. By \Cref{4.18}, $a^{\smallfrown} (x \upharpoonright l)$ is either strongly mixed or strongly separated by $Y$, finishing the induction. \qedhere

\end{proof}

\begin{lemma} \label{4.20}

Let $a,b \in (\hat{\mathcal{F}} \setminus \mathcal{F})|Y$ and assume that $a$ and $b$ are mixed by $Y$. Furthermore, assume that both $a$ and $b$ are separated in some sense by $Y$. Then either $a$ and $b$ have the same separation type, or $a$ is minmax-separated by $Y$ and $b$ is min-separated by $Y$, or vice versa. Moreover, if $b$ is minmax-separated by $Y$ and $a$ is min-separated by $Y$, then there is $a <_b y \leq Y$ such that $a^{\smallfrown}y \in \mathcal{F}$ and there is $1\leq k < |y|$ for which $a^{\smallfrown} (y \upharpoonright k)$ is separated in some sense by $Y$. Finally, in this case, for all $a <_b y \leq Y$ such that $a^{\smallfrown}y \in \mathcal{F}$, there is some $1\leq k < |y|$ for which $a^{\smallfrown} (y \upharpoonright k)$ is separated in some sense by $Y$, and $a^{\smallfrown} (y \upharpoonright k_0)$ is max-separated by $Y$, where $k_0$ is the least such $k$.

\end{lemma}

\begin{proof}

The first assertion directly follows from \Cref{4.18}. Now let us assume that $a,b \in (\hat{\mathcal{F}} \setminus \mathcal{F})|Y$ are mixed by $Y$, $b$ is minmax-separated by $Y$ and $a$ is min-separated by $Y$. Take $a,b <_b x,y$ with $g(b^{\smallfrown}x)=g(a^{\smallfrown}y)$. Assume towards a contradiction that for all $1 \leq k \leq |y|$, $a^{\smallfrown}(y \upharpoonright k)$ is strongly mixed by $Y$. It follows that $a^{\smallfrown} y(0)$ and $a^{\smallfrown}y$ are mixed by $Y$. Take $s \in [Y/a]$ such that $\operatorname{min-sep}(s) = \operatorname{min-sep}(y(0))$ but $\operatorname{max-sep}(s) >_b \operatorname{max-sep}(x(0))$. By definition of mixing, $b^{\smallfrown}x$ and $a^{\smallfrown}y$ are mixed by $Y$. Since $a$ is min-separated by $Y$, $a^{\smallfrown}y(0)$ and $b^{\smallfrown}m$ are mixed by $Y$. It follows by transitivity of mixing that $a^{\smallfrown}s$ and $b^{\smallfrown}x$ are mixed by $Y$. Find $s,x <_b z \leq Y$ such that $a^{\smallfrown}s^{\smallfrown}z \in \mathcal{F}|Y$ and $g(b^{\smallfrown}x) = g(a^{\smallfrown}s^{\smallfrown}z)$. This contradicts \Cref{4.17} since $\operatorname{max-sep}(s) >_b \operatorname{max-sep}(x(0))$.

To finish, assume that $a,b \in (\hat{\mathcal{F}} \setminus \mathcal{F})|Y$ are mixed by $Y$, $b$ is minmax-separated by $Y$ and $a$ is min-separated by $Y$. Take $a,b<x,y$ with $g(b^{\smallfrown}x)=g(a^{\smallfrown}y)$. By the preceding paragraph and \ref{canon4} of canonical, there is a least $1 \leq k_0 < |y|$ such that $a^{\smallfrown} (y \upharpoonright k_0)$ is separated in some sense by $Y$. By \ref{canon5} of canonical, it suffices to show that $a^{\smallfrown} (y \upharpoonright k_0)$ is max-separated by $Y$. Indeed, since $g(b^{\smallfrown}x)=g(a^{\smallfrown}y)$, by \Cref{4.17}, we know that $\operatorname{min-sep}(x(0)) = \operatorname{min-sep}(y(0))$ and $\operatorname{max-sep}(x(0)) > \operatorname{max-sep}(y(0))$. Find $j \in \omega$ and $A, B \in [Y]$, $A \neq \varnothing$ such that $x(0) = Y(j) \cup A$ and $y(0) = Y(j) \cup B$. It follows by \Cref{4.17} that $b^{\smallfrown} Y(j)\uparrow$ and $a^{\smallfrown} Y(j) \uparrow$ are mixed by $Y$. By \Cref{4.12} part \ref{4.12ii}, $a^{\smallfrown} Y(j) \uparrow$ and $a^{\smallfrown} y(0)$ are mixed by $Y$. Finally, by \Cref{4.12} part \ref{4.12i}, we see that $a^{\smallfrown} y(0)$ and $a^{\smallfrown} (y \upharpoonright k_0)$ are mixed by $Y$. It follows that $b^{\smallfrown} Y(j) \uparrow$ and $a^{\smallfrown} (y \upharpoonright k_0)$ are mixed by $Y$. But now, by \Cref{4.11}, $b^{\smallfrown} Y(j) \uparrow$ is max-separated by $Y$. Since $a^{\smallfrown} (y \upharpoonright k_0)$ is separated in some sense by $Y$, it follows from \Cref{4.18} that $a^{\smallfrown} (y \upharpoonright k_0)$ is max-separated by $Y$. \qedhere

\end{proof}

We now define the parameter function $\gamma : (\hat{\mathcal{F}}\setminus \mathcal{F})|Y \to \{\operatorname{sm}, \operatorname{min-sep},$ $\operatorname{max-sep},$ $\operatorname{minmax-sep}, \operatorname{sss}, \operatorname{vss}\}$:

\begin{definition}\label{parameter}

For $a \in (\hat{\mathcal{F}}\setminus \mathcal{F})|Y$, we let $\gamma(a) = \operatorname{sm}$ if $a$ is strongly mixed by $Y$, $\gamma(a) = \operatorname{min-sep}$ if $a$ is min-separated by $Y$, $\gamma(a) = \operatorname{max-sep}$ if $a$ is max-separated by $Y$, $\gamma(a) = \operatorname{minmax-sep}$ if $a$ is minmax-separated by $Y$, $\gamma(a) = \operatorname{sss}$ if $a$ is still strongly separated by $Y$, and $\gamma(a) = \operatorname{vss}$ if $a$ is very strongly separated by $Y$.

\end{definition}

Note that by \Cref{4.19} and \ref{canon5} of canonical, $\gamma$ is admissible. Let us assume that $\gamma$ is not trivial; i.e., assume that there is $a \in (\hat{\mathcal{F}}\setminus \mathcal{F})|Y$ with $\gamma(a) \neq \operatorname{sm}$. It then follows that $\Gamma_{\gamma}(b) \neq \varnothing$ for all $b \in \mathcal{F}|Y$. $Y \leq X$ and $\gamma$ being defined, we can start the proof of the theorem:

\begin{proof}[Proof of \Cref{canonization}]

Let $\mathcal{F}$ be a front on some $X \in \mathrm{FIN}^{[\infty]}$, and let $g : \mathcal{F} \to \omega$ be a function. Pick a canonical $Y \leq X$ for $g$. The following four claims will finish the proof:

\begin{theorem1} \label{Claim1}

Let $a, a' \in \mathcal{F}|Y$. Take $k \leq |a|$ and $k' \leq |a'|$ with $\Gamma_{\gamma}(a \upharpoonright k) = \Gamma_{\gamma}(a' \upharpoonright k')$, and assume it is not the case that $\gamma(a \upharpoonright l) = \operatorname{sss}$ and $\gamma(a' \upharpoonright l') = \operatorname{vss}$ (or vice versa) where $l\leq k-1$ and $l' \leq k'-1$ are maximal with $\gamma(a \upharpoonright l), \gamma(a' \upharpoonright l') \neq \operatorname{sm}$. Then $a \upharpoonright k$ and $a' \upharpoonright k'$ are mixed by $Y$.

\end{theorem1}

\begin{proof}

Let $a, a' \in \mathcal{F}|Y$, take $k \leq |a|$ and $k' \leq |a'|$ with $\Gamma_{\gamma}(a \upharpoonright k) = \Gamma_{\gamma}(a' \upharpoonright k')$, and assume it is not the case that $\gamma(a \upharpoonright l) = \operatorname{sss}$ and $\gamma(a' \upharpoonright l') = \operatorname{vss}$ (or vice versa) where $l\leq k-1$ and $l' \leq k'-1$ are maximal with $\gamma(a \upharpoonright l), \gamma(a' \upharpoonright l') \neq \operatorname{sm}$. Let us prove the result by induction on $|\Gamma_{\gamma}(a \upharpoonright k)| = |\Gamma_{\gamma}(a' \upharpoonright k')|$.

The case when $|\Gamma_{\gamma}(a \upharpoonright k)| = 0 = |\Gamma_{\gamma}(a' \upharpoonright k')|$ follows by successive applications of \Cref{4.12} part \ref{4.12i}.

Now let us assume that $|\Gamma_{\gamma}(a \upharpoonright k)| = |\Gamma_{\gamma}(a' \upharpoonright k')| > 0$ and the result holds for all $b \upharpoonright l$, $b' \upharpoonright l' \in \hat{\mathcal{F}} | Y$ with $\Gamma_{\gamma}(b \upharpoonright l) = \Gamma_{\gamma}(b' \upharpoonright l')$ and $|\Gamma_{\gamma}(b \upharpoonright l)| < |\Gamma_{\gamma}(a \upharpoonright k)|$. Find maximal $k_0 < k$ and $k_0' < k'$ with $|\Gamma_{\gamma}(a \upharpoonright k_0)| = |\Gamma_{\gamma}(a \upharpoonright k)|-1$ and $|\Gamma_{\gamma}(a' \upharpoonright k_0')| = |\Gamma_{\gamma}(a' \upharpoonright k')|-1$. It follows that $\Gamma_{\gamma}(a \upharpoonright k_0) = \Gamma_{\gamma}(a' \upharpoonright k_0')$, and so the induction hypothesis implies that $a \upharpoonright k_0$ and $a' \upharpoonright k_0'$ are mixed by $Y$. Recall that the images of $\operatorname{min-sep}$, $\operatorname{max-sep}$, $\operatorname{minmax-sep}$ and $\operatorname{sss}=\operatorname{vss}$ on $[Y]$ are disjoint. Let $d=|\Gamma_{\gamma}(a \upharpoonright k)| = |\Gamma_{\gamma}(a' \upharpoonright k')|$. We split into cases.

First, assume that $\Gamma_{\gamma}(a \upharpoonright k)(d-1) = \Gamma_{\gamma}(a' \upharpoonright k')(d-1)$ is in the range of $\operatorname{min-sep}$. Since $k_0$ and $k_0'$ were maximal, it follows that $a \upharpoonright k_0$ and $a' \upharpoonright k_0'$ are both min-separated by $Y$ and $\operatorname{min-sep}(a(k_0)) = \operatorname{min-sep}(a'(k'_0))$. It follows by \Cref{4.16} part \ref{4.16i} that $(a \upharpoonright k_0)^{\smallfrown} a(k_0)$ and $(a' \upharpoonright k_0')^{\smallfrown} a(k_0')$ are mixed by $Y$. Since all of the proper initial segments of $a \upharpoonright k$ extending 
$(a \upharpoonright k_0)^{\smallfrown} a(k_0)$, and all of the proper initial segments of $a' \upharpoonright k'$ extending 
$(a' \upharpoonright k'_0)^{\smallfrown} a'(k'_0)$ are strongly mixed by $Y$, by repeated applications of \Cref{4.12} part \ref{4.12i}, we get that $a \upharpoonright k$ and $a' \upharpoonright k'$ are mixed by $Y$.

If $\Gamma_{\gamma}(a \upharpoonright k)(d-1) = \Gamma_{\gamma}(a' \upharpoonright k')(d-1)$ is in the range of $\operatorname{max-sep}$, then the claim similarly follows by \Cref{4.16} part \ref{4.16ii}.

Assume that $\Gamma_{\gamma}(a \upharpoonright k)(d-1) = \Gamma_{\gamma}(a' \upharpoonright k')(d-1)$ is in the range of $\operatorname{minmax-sep}$. Since $k_0$ and $k_0'$ were maximal we have four alternatives:

If $a \upharpoonright k_0$ and $a' \upharpoonright k_0'$ are both minmax-separated by $Y$, then the proof works similarly to the previous two cases.

Assume that $a \upharpoonright k_0$ is minmax-separated by $Y$ and $a' \upharpoonright k_0'$ is min-separated by $Y$, and the least $k_0' < l' < k'$ with $\gamma(a' \upharpoonright l') \neq \operatorname{sm}$ satisfies $\gamma(a' \upharpoonright l') = \operatorname{max-sep}$. Since $\Gamma_{\gamma}(a \upharpoonright k)(d-1) = \Gamma_{\gamma}(a' \upharpoonright k')(d-1)$, we know that $\operatorname{min-sep}(a(k_0)) = \operatorname{min-sep}(a'(k'_0))$. It follows that $a(k_0) = Y'(j) \cup A$ and $a'(k'_0) = Y'(j) \cup B$ for some $j \in \omega$ and for some $A, B \in [Y'/Y'(j)]$ with $A,B \neq \varnothing$. It follows by \Cref{4.17} that $(a \upharpoonright k_0)^{\smallfrown} Y'(j)\uparrow$ and $(a' \upharpoonright k'_0)^{\smallfrown} Y'(j)\uparrow$ are mixed by $Y'$. Since $a' \upharpoonright k'_0$ is min-separated by $Y'$, by \Cref{4.11} we may conclude that $(a' \upharpoonright k'_0)^{\smallfrown} Y'(j)\uparrow$ is strongly mixed by $Y'$, and it then follows by \Cref{4.12} part \ref{4.12i} that $(a' \upharpoonright k'_0)^{\smallfrown} Y'(j)\uparrow$ and $(a' \upharpoonright k'_0)^{\smallfrown} a'(k_0')$ are mixed by $Y'$. Again by \Cref{4.12} part \ref{4.12i}, we see that $(a' \upharpoonright k'_0)^{\smallfrown} a'(k_0')$ and $a' \upharpoonright l'$ are mixed by $Y'$. Finally, since by \Cref{4.11} $(a \upharpoonright k_0)^{\smallfrown} Y'(j)\uparrow$ is max-separated by $Y'$, $(a \upharpoonright k_0)^{\smallfrown} Y'(j)\uparrow$ and $a' \upharpoonright l'$ being mixed by $Y'$ and the fact that $\operatorname{max-sep}(a(k_0)) = \operatorname{max-sep}(a'(l'))$, by \Cref{4.16} part \ref{4.16ii}, implies that $(a \upharpoonright k_0)^{\smallfrown} a(k_0)$ and $(a' \upharpoonright l')^{\smallfrown} a'(l')$ are mixed by $Y'$. It then follows by successive applications of \Cref{4.12} part \ref{4.12i} that $a \upharpoonright k$ and $a' \upharpoonright k'$ are mixed by $Y'$, and hence also by $Y$.

Now assume that $a \upharpoonright k_0$ is min-separated by $Y$, and the least $k_0 < l < k$ with $\gamma(a \upharpoonright l) \neq \operatorname{sm}$ satisfies $\gamma(a \upharpoonright l) = \operatorname{max-sep}$ and also $a' \upharpoonright k_0'$ is min-separated by $Y$, and the least $k_0' < l' < k'$ with $\gamma(a' \upharpoonright l') \neq \operatorname{sm}$ satisfies $\gamma(a' \upharpoonright l') = \operatorname{max-sep}$. It follows by \Cref{4.16} part \ref{4.16i} that $(a \upharpoonright k_0)^{\smallfrown} a(k_0)$ and $(a' \upharpoonright k_0')^{\smallfrown} a(k_0')$ are mixed by $Y$ and so $a \upharpoonright l$ and $a' \upharpoonright l'$ are mixed by $Y$ as well. Since $\operatorname{max-sep}(a(l)) = \operatorname{max-sep}a'(l'))$, we get the result.

Finally, assume that $\Gamma_{\gamma}(a \upharpoonright k)(d-1) = \Gamma_{\gamma}(a' \upharpoonright k')(d-1)$ is in the range of $\operatorname{sss}=\operatorname{vss}$. Since $k_0$ and $k_0'$ were maximal, it follows that $a \upharpoonright k_0$ and $a' \upharpoonright k_0'$ are both strongly separated by $Y$. We shall consider three cases.

First, assume that both $a \upharpoonright k_0$ and $a' \upharpoonright k_0'$ are very strongly separated by $Y$. By construction it follows that $\Gamma_{\gamma}(a \upharpoonright k)(d-1) = a(k_0)$ and $\Gamma_{\gamma}(a' \upharpoonright k')(d-1) = a(k'_0)$. Similarly to the above arguments, we see that $a \upharpoonright k$ and $(a \upharpoonright k_0)^{\smallfrown} a(k_0)$ as well as $a' \upharpoonright k'$ and $ (a' \upharpoonright k_0')^{\smallfrown} a(k_0')$ are mixed by $Y$. Since $a(k_0) = a'(k_0')$, by \Cref{4.16} part \ref{4.16v} we conclude that $a \upharpoonright k$ and $a' \upharpoonright k'$ are mixed by $Y$.

Second, assume that one of $a \upharpoonright k_0$ and $a' \upharpoonright k_0'$ is still strongly separated by $Y$ and the other is very strongly separated by $Y$. Without loss of generality, we may assume that $a \upharpoonright k_0$ is still strongly separated by $Y$ and $a' \upharpoonright k_0'$ is very strongly separated by $Y$. It follows that $\Gamma_{\gamma}(a \upharpoonright k)(d-1) = \Gamma_{\gamma}(a' \upharpoonright k')(d-1) = a'(k_0')$. By \Cref{4.19}, we define a sequence $\{k_i\}_{i <N}$ as follows: Let $k_i > k_0$ be given and find the least $k_{i+1} > k_i$ with $\gamma(a \upharpoonright k_{i+1}) = \operatorname{sss}$ or $\operatorname{vss}$. If $\gamma(a \upharpoonright k_{i+1}) = \operatorname{vss}$, stop; if $\gamma(a \upharpoonright k_{i+1}) = \operatorname{sss}$, then similarly find the least $k_{i+2} > k_{i+1}$ with  $\gamma(a \upharpoonright k_{i+2}) = \operatorname{sss}$ or $\operatorname{vss}$. At the end, we find an $N<\omega$, and obtain a sequence $(k_i)_{1\leq i \leq N}$ for some $k_N \leq k-1$ (by the assumption on $k$ and $k'$) such that $\gamma(a \upharpoonright k_i) = \operatorname{sss}$ for all $i < N$ and $\gamma(a \upharpoonright k_N) = \operatorname{vss}$. It follows that $\bigcup_{0\leq i \leq N} a(k_i) = a'(k_0')$. \Cref{4.17new} part \ref{4.17newii} implies that $(a \upharpoonright k_0)^{\smallfrown}a(k_0)$ and $(a' \upharpoonright k_0')^{\smallfrown} a(k_0)\uparrow$ are mixed by $Y$. By \Cref{4.12} part \ref{4.12i}, $(a \upharpoonright k_0)^{\smallfrown}a(k_0)$ and $a \upharpoonright k_1$ are mixed by $Y$. It follows that $a \upharpoonright k_1$ and $(a' \upharpoonright k_0')^{\smallfrown} a(k_0)\uparrow$ are also mixed by $Y$. Since $a' \upharpoonright k_0'$ is very strongly separated, $(a' \upharpoonright k_0')^{\smallfrown} a(k_0)\uparrow$ is very strongly separated as well. In general, given $a \upharpoonright k_i$ and $(a' \upharpoonright k_0')^{\smallfrown} \bigcup_{j<i} a(k_j)\uparrow$ mixed by $Y$, we can do the same argument to see that $a \upharpoonright k_{i+1}$ and $(a' \upharpoonright k_0')^{\smallfrown} \bigcup_{j <i+1} a(k_j)\uparrow$ are mixed by $Y$. At the end we will see that $a \upharpoonright k_N$ and $(a' \upharpoonright k_0')^{\smallfrown} (\bigcup_{0\leq i < N} a(k_i))\uparrow$ are mixed by $Y$, and $a \upharpoonright k_N$ and $(a' \upharpoonright k_0')^{\smallfrown} (\bigcup_{0\leq i < N} a(k_i))\uparrow$ are both very strongly separated by $Y$. It follows by \Cref{4.16} part \ref{4.16v} that $(a \upharpoonright k_N)^{\smallfrown} a(k_N)$ and $(a' \upharpoonright k_0')^{\smallfrown} (\bigcup_{0\leq i < N} a(k_i) \cup a(k_N)) = (a' \upharpoonright k_0')^{\smallfrown} a'(k_0')$ are mixed by $Y$, and we get the result by \Cref{4.12} part \ref{4.12i}.

To finish, assume that both $a \upharpoonright k_0$ and $a' \upharpoonright k_0'$ are still strongly separated by $Y$. Repeating the construction from the previous paragraph, we get sequences $(k_i)_{0\leq i \leq N}$ and $(k'_i)_{0\leq i \leq N'}$ for some $k_N \leq k-1$ and $k'_{N'} \leq k'-1$ such that $\gamma(a \upharpoonright k_i) = \operatorname{sss}$ for all $i < N$, and $\gamma(a \upharpoonright k_N) = \operatorname{sss}$ or $\operatorname{vss}$; and similarly $\gamma(a' \upharpoonright k'_i) = \operatorname{sss}$ for all $i < N'$, and $\gamma(a' \upharpoonright k'_{N'}) = \operatorname{sss}$ or $\operatorname{vss}$. Note by the assumption on $k$ and $k'$ that $\gamma(a \upharpoonright k_N) = \operatorname{sss}$ if and only if $\gamma(a' \upharpoonright k'_{N'}) = \operatorname{sss}$. By the assumption, $\bigcup_{0\leq i \leq N} a(k_i) = \bigcup_{0\leq i \leq N'} a'(k'_i)$.

We claim that $(a \upharpoonright k_{N})^{\smallfrown} a(k_N)$ and $(a' \upharpoonright {k'}_{N'})^{\smallfrown} a({k'}_{N'})$ are mixed by $Y$. Indeed, find the least $0\leq i_0 \leq N$ with $a(k_{i_0}) \neq a'(k'_{i_0})$ (if $N = N'$ and $a(k_i) = a'(k'_i)$ for all $i < N$, then it follows from successive applications of \Cref{4.17new} part \ref{4.17newii} that $a \upharpoonright k_N$ and $a' \upharpoonright k'_{N'}$ are mixed by $Y$, and the respective part of \Cref{4.17} or \Cref{4.16} proves the claim). Note that $i_0 < N$. Since $\bigcup_{0\leq i \leq N} a(k_i) = \bigcup_{0\leq i \leq N'} a'(k'_i)$, either $a(k_{i_0})$ is a proper initial segment of $a'({k'}_{i_0})$ or vice versa. Without loss of generality, assume that $a(k_{i_0})$ is a proper initial segment of $a'({k'}_{i_0})$. It follows by \Cref{4.17new} part \ref{4.17newi} that $(a \upharpoonright k_{i_0})^{\smallfrown} a(k_{i_0})$ and $(a' \upharpoonright {k'}_{i_0})^{\smallfrown} a({k}_{i_0})\uparrow$ are mixed by $Y$. Note that $(a' \upharpoonright {k'}_{i_0})^{\smallfrown} a({k}_{i_0})\uparrow$ is still strongly separated by $Y$. Find the least $i_1 > i_0$ with $a'(k'_{i_0}) \subseteq \bigcup_{0\leq i \leq i_1} a(k_i)$. If $a'(k'_{i_0}) = \bigcup_{0\leq i \leq i_1} a(k_i)$, let $v = a'(k'_{i_0}) \setminus a(k_{i_0})$; otherwise let $v = (a'(k'_{i_0}) \setminus a(k_{i_0}))\uparrow$. In either case, $(a \upharpoonright k_{i_0})^{\smallfrown}a(k_{i_0})^{\smallfrown} v$ and $(a' \upharpoonright {k'}_{i_0})^{\smallfrown} a'(k'_{i_0})$ are still strongly separated by $Y$; also $(a \upharpoonright k_{i_0})^{\smallfrown}a(k_{i_0})^{\smallfrown} v$ and $((a' \upharpoonright {k'}_{i_0})^{\smallfrown}a({k}_{i_0})\uparrow)^{\smallfrown}(a'(k'_{i_0}) \setminus a(k_{i_0})) = (a' \upharpoonright {k'}_{i_0})^{\smallfrown} a'(k'_{i_0})$ are mixed by $Y$. Repeating this argument, by $\bigcup_{0\leq i \leq N} a(k_i)$ $= \bigcup_{0\leq i \leq N'} a'(k'_i)$, we without loss of generality see that $(a \upharpoonright k_N)^{\smallfrown} w \uparrow$ and $(a' \upharpoonright {k'}_{N'})$ are mixed by $Y$ for some $a \upharpoonright k_N < w \in [Y]$ with $a(k_N) \setminus w = a'({k'}_{N'})$ (we allow $w = \varnothing$ with the convention ``$\varnothing \uparrow = \varnothing$"). Note that either both $(a \upharpoonright k_N)^{\smallfrown} w \uparrow$ and $(a' \upharpoonright {k'}_{N'})$ are still strongly separated by $Y$ or both $(a \upharpoonright k_N)^{\smallfrown} w \uparrow$ and $(a' \upharpoonright {k'}_{N'})$ are very strongly separated by $Y$. In the former case, since $w \cup a'({k'}_{N'}) = a(k_N)$, we see by \Cref{4.17new} part \ref{4.17newi} that $(a \upharpoonright k_{N})^{\smallfrown} a(k_N)$ and $(a' \upharpoonright {k'}_{N'})^{\smallfrown} a({k'}_{N'})$ are mixed by $Y$. In the latter case, this time by \Cref{4.16} part \ref{4.16v}, we see that $(a \upharpoonright k_{N})^{\smallfrown} a(k_N)$ and $(a' \upharpoonright {k'}_{N'})^{\smallfrown} a({k'}_{N'})$ are mixed by $Y$, and we are done. \qedhere

\end{proof}

\begin{theorem1} \label{Claim2}

Let $a, a' \in \mathcal{F} | Y$. If $\Gamma_{\gamma}(a) = \Gamma_{\gamma}(a')$, then $g(a) = g(a')$.

\end{theorem1}

\begin{proof}

Let $k = |a|$ and take maximal $k_0 <k$ with $\gamma(a \upharpoonright k_0) \neq \operatorname{sm}$. Assume towards a contradiction that $\gamma(a \upharpoonright k_0) = \operatorname{sss}$. It follows that $(a \upharpoonright k_0)^{\smallfrown} a(k_0)$ and $(a \upharpoonright k_0)^{\smallfrown} (a(k_0)\uparrow)$ are mixed. Take $(a \upharpoonright k_0)^{\smallfrown} a(k_0)< x,y$ (possibly empty) with $g((a \upharpoonright k_0)^{\smallfrown} a(k_0)^{\smallfrown}x) = g((a \upharpoonright k_0)^{\smallfrown} (a(k_0)\uparrow)^{\smallfrown}y)$. Since $k_0<k$ were maximal, it follows by \Cref{4.12} part \ref{4.12i} that $(a \upharpoonright k_0)^{\smallfrown}a(k_0)$ and $(a \upharpoonright k_0)^{\smallfrown}a(k_0)^{\smallfrown} x$ are mixed by $Y$. Similarly, $(a \upharpoonright k_0)^{\smallfrown}(a(k_0) \cup y(0))$ and $(a \upharpoonright k_0)^{\smallfrown}(a(k_0)\uparrow)^{\smallfrown} y$ are also mixed by $Y$. Hence, $(a \upharpoonright k_0)^{\smallfrown}a(k_0)$ and $(a \upharpoonright k_0)^{\smallfrown}(a(k_0)\cup y(0))$ are mixed by $Y$, which contradicts the fact that $a \upharpoonright k_0$ is strongly separated by $Y$.

It follows that \Cref{Claim1} is applicable. Therefore, $a = a \upharpoonright k$ and $a' = a' \upharpoonright k'$ are mixed by $Y$. By definition of mixing, since $a, a' \in \mathcal{F} | Y$, it follows that $g(a) = g(a')$. \qedhere

\end{proof}

\begin{theorem1} \label{Claim3}

Let $a,a' \in \mathcal{F}|Y$. Then $\Gamma_{\gamma}(a) \not\sqsubset \Gamma_{\gamma}(a')$.

\end{theorem1}

\begin{proof}

This directly follows by the observation that no $a \in \mathcal{F} | Y$ can be mixed with some $b \in (\hat{\mathcal{F}} \setminus \mathcal{F}) | Y$ which is separated in some sense by $Y$. Indeed, if there is $b <_b y \leq Y$ such that $g(a) = g(b^{\smallfrown}y)$, it follows that $b$ and $b^{\smallfrown}y(0)$ are mixed by $Y$, contradicting the fact that $b$ is separated in some sense by $Y$. \qedhere

\end{proof}

\begin{theorem1} \label{Claim4}

Let $a, a' \in \mathcal{F} | Y$. If $\Gamma_{\gamma}(a) \neq \Gamma_{\gamma}(a')$, then $g(a) \neq g(a')$.

\end{theorem1}

\begin{proof}

Let $\Gamma_{\gamma}(a) \neq \Gamma_{\gamma}(a')$ and assume towards a contradiction that $g(a) = g(a')$. By \Cref{Claim3}, choose $l < |\Gamma_{\gamma}(a)|,|\Gamma_{\gamma}(a')|$ maximal with $\Gamma_{\gamma}(a) \upharpoonright l = \Gamma_{\gamma}(a') \upharpoonright l$ (we allow $l = 0$). Choose maximal $k<|a|$ and $k'<|a'|$ with $\Gamma_{\gamma}(a \upharpoonright k) = \Gamma_{\gamma}(a) \upharpoonright l$ and $\Gamma_{\gamma}(a' \upharpoonright k') = \Gamma_{\gamma}(a') \upharpoonright l$. Note that both $a \upharpoonright k$ and $a' \upharpoonright k'$ are separated in some sense by $Y$. By definition of $\Gamma_{\gamma}$, the hypothesis of \Cref{Claim1} is satisfied for $a \upharpoonright k$ and $a' \upharpoonright k'$. It follows that $a \upharpoonright k$ and $a' \upharpoonright k'$ are mixed by $Y$. By \Cref{4.20}, we only have the following cases:

First, assume that both $a \upharpoonright k$ and $a' \upharpoonright k'$ are min-separated by $Y$. By \Cref{4.16} part \ref{4.16i}, it follows that $\operatorname{min-sep}(a(k)) = \operatorname{min-sep}(a'(k'))$. Since $l$ was chosen to be maximal, without loss of generality there is minimal $k<j$ such that $a \upharpoonright j$ is separated in some sense by $Y$. It follows from the above proof that there is also minimal $k'<j'$ such that $a' \upharpoonright j'$ is separated in some sense by $Y$. By construction of $\Gamma_{\gamma}$, both $a \upharpoonright j$ and $a' \upharpoonright j'$ are max-separated by $Y$. It follows that $a \upharpoonright j$ and $a' \upharpoonright j'$ are mixed by $Y$. By \Cref{4.16} part \ref{4.16ii}, we get $\operatorname{max-sep}(a(j)) = \operatorname{max-sep}(a'(j'))$, which contradicts the maximality of $l$.

The cases when either $a \upharpoonright k$ and $a' \upharpoonright k'$ are both max-separated by $Y$ or are both minmax-separated by $Y$ reach a contradiction by the respective parts of \Cref{4.16}.

Assume that $a \upharpoonright k$ is min-separated by $Y$ and $a' \upharpoonright k'$ is minmax-separated by $Y$. It follows by \Cref{4.17} that $\operatorname{min-sep}(a(k)) = \operatorname{min-sep}(a'(k'))$, and $\operatorname{max-sep}(a(k))$ $<_b\operatorname{max-sep}(a'(k'))$. By \Cref{4.20}, there is minimal $k < j$ such that $a \upharpoonright j$ is separated in some sense by $Y$, and actually $a \upharpoonright j$ is max-separated by $Y$. By the argument in the proof of \Cref{4.20}, we in fact see that $a \upharpoonright j$ and $(a' \upharpoonright k')^{\smallfrown} Y(j_0)\uparrow$ are mixed by $Y$, where $j_0 \in \omega$ satisfies $\operatorname{min-sep}(a'(k')) \in Y(j_0)$. Since $(a' \upharpoonright k')^{\smallfrown} Y(j_0)\uparrow$ is max-separated by $Y$, by \Cref{4.16} part \ref{4.16ii}, it follows that $\operatorname{max-sep}(a(j)) = \operatorname{max-sep}(a'(k'))$, which contradicts the maximality of $l$.

Assume that both $a \upharpoonright k$ and $a' \upharpoonright k'$ are strongly separated by $Y$. By \Cref{4.16} part \ref{4.16iv}, either $a(k)$ is an initial segment of $a'(k')$ or $a'(k')$ is an initial segment of $a(k)$. By \Cref{4.16} part \ref{4.16v}, at least one of  $a \upharpoonright k$ or $a' \upharpoonright k'$ has to be still strongly separated by $Y$.

Assume that $a \upharpoonright k$ is still strongly separated by $Y$ and $a' \upharpoonright k'$ is very strongly separated by $Y$. We cannot have $a(k) = a'(k')$, since it would imply that $(a \upharpoonright k)^{\smallfrown} a(k)$ and $(a' \upharpoonright k')^{\smallfrown} a(k)$ are mixed by $Y$, contradicting \Cref{4.17new} part \ref{4.17newii}. Again by \Cref{4.17new} part \ref{4.17newii}, $a'(k')$ cannot be a proper initial segment of $a(k)$. It follows that $a(k)$ is a proper initial segment of $a'(k')$, and so $(a \upharpoonright k)^{\smallfrown}a(k)$ and $(a' \upharpoonright k')^{\smallfrown}a(k)\uparrow$ are mixed by $Y$. Find minimal $k<j_0$ such that $a \upharpoonright j_0$ is separated in some sense by $Y$. By \Cref{4.19}, $a \upharpoonright j_0$ is strongly separated by $Y$. Note that $(a \upharpoonright k)^{\smallfrown}a(k)$ and $a \upharpoonright j_0$ are mixed by $Y$, and $(a' \upharpoonright k')^{\smallfrown}a(k)\uparrow$ is very strongly separated by $Y$. If $a \upharpoonright j_0$ is very strongly separated by $Y$, then $a(j_0) = a'(k') \setminus a(k)$ by \Cref{4.16} part \ref{4.16v}, and we contradict the maximality of $l$. If $a \upharpoonright j_0$ is still strongly separated by $Y$, then we can similarly find minimal $j_0 < j_1$ such that $a \upharpoonright j_1$ is strongly separated by $Y$. In general, suppose that $j_i$ is given with $a \upharpoonright j_i$ strongly separated by $Y$, $a \upharpoonright j_i$ and $(a' \upharpoonright k')^{\smallfrown}(a(k) \cup \bigcup_{i' <i} a(j_{i'}))$ mixed by $Y$, and $a(k) \cup \bigcup_{i' <i} a(j_{i'}) \sqsubset a'(k')$. If $a \upharpoonright j_i$ is very strongly separated by $Y$, then by the above argument $a'(k') \setminus (a(k) \cup \bigcup_{i' <i} a(j_{i'})) = a(j_i)$, and we contradict the maximality of $l$. If, on the other hand, $a \upharpoonright j_i$ is still strongly separated by $Y$, then we find the least $j_i < j_{i+1}$ such that $a \upharpoonright j_{i+1}$ is separated in some sense by $Y$. By the argument in the base case, $a(j_i)$ is a proper initial segment of $a'(k') \setminus (a(k) \cup \bigcup_{i' <i} a(j_{i'}))$. At the end we get some $N$ with $a \upharpoonright j_N$ and $(a' \upharpoonright k')^{\smallfrown}(a(k) \cup \bigcup_{i' <N} a(j_{i'}))$ mixed by $Y$, and $a \upharpoonright j_N$ is very strongly separated by $Y$. This contradicts the maximality of $l$, as in the previous cases.

Assume that both $a \upharpoonright k$ and $a' \upharpoonright k'$ are still strongly separated by $Y$. By \Cref{4.16} part \ref{4.16iv}, either $a(k)$ is an initial segment of $a'(k')$ or $a'(k')$ is an initial segment of $a(k)$. Let $k < j_0 < j_1 < \ldots < j_N$ increasingly enumerate the natural numbers such that $a \upharpoonright j_i$ is still strongly separated by $Y$ for all $i < N$ and $a \upharpoonright j_N$ is very strongly separated by $Y$. Similarly, let $k' < j'_0 < j'_1 < \ldots < j'_{N'}$ increasingly enumerate the natural numbers such that $a' \upharpoonright j'_i$ is still strongly separated by $Y$ for all $i < N'$ and $a \upharpoonright j'_{N'}$ is very strongly separated by $Y$. If $a(k) = a'(k')$, then we see that $(a \upharpoonright k)^{\smallfrown}a(k)$ and $(a' \upharpoonright k')^{\smallfrown}a'(k')$ are mixed by $Y$, which implies that $a \upharpoonright j_0$ and $a' \upharpoonright j'_0$ are also mixed by $Y$, and the argument will be similar to the next case. Suppose now that $a(k)$ is a proper initial segment of $a'(k')$. Then $(a \upharpoonright k)^{\smallfrown}a(k)$ and $(a' \upharpoonright k')^{\smallfrown}a(k)\uparrow$ are mixed by $Y$, and so $a \upharpoonright j_0$ and $(a' \upharpoonright k')^{\smallfrown}a(k)\uparrow$ are also mixed by $Y$. Note that $(a' \upharpoonright k')^{\smallfrown}a(k)\uparrow$ is still strongly separated by $Y$. It follows that either $a(j_0)$ is an initial segment of $a'(k') \setminus a(k)$, or vice versa.

Observe that, if $a \upharpoonright j_0$ is very strongly separated by $Y$ (i.e., if $N =0$), then $a'(k') \setminus a(k)$ has to be a proper initial segment of $a(j_0)$ (otherwise we contradict \Cref{4.17new} part \ref{4.17newii}). We see that $(a \upharpoonright j_0)^{\smallfrown}a'(k')\uparrow$ and $a' \upharpoonright j'_0$ are mixed by $Y$. Note that $(a \upharpoonright j_0)^{\smallfrown}a'(k')\uparrow$ is very strongly separated by $Y$. If $a' \upharpoonright j'_0$ is very strongly separated by $Y$, then by \Cref{4.16} part \ref{4.16v} $a(j_0) \setminus a'(k') = a'(j'_0)$, and we contradict the maximality of $l$. If, on the other hand, $a' \upharpoonright j'_0$ is still strongly separated by $Y$, then it follows similarly that $a'(j'_0)$ has to be a proper initial segment of $a(j_0) \setminus a'(k')$. At the end of finitely many steps, we reach $a' \upharpoonright j'_{N'}$, which is very strongly separated by $Y$. Also, $(a \upharpoonright j_0)^{\smallfrown} a'(j'_{N'-1})\uparrow$ and $a' \upharpoonright j'_{N'}$ are mixed by $Y$ and we have $a(j_0) \setminus a'(j'_{N'-1}) = a'(j_{N'})$, which again contradicts the maximality of $l$.

Finally, if $a \upharpoonright j_0$ is still strongly separated by $Y$, then we repeat the same argument until we without loss of generality reach $i=N$ and $a \upharpoonright j_N$, which is very strongly separated by $Y$. Then the argument in the preceding paragraph applied to $a \upharpoonright j_N$ contradicts the maximality of $l$. \qedhere

\end{proof}

We now apply \Cref{Claim2} and \Cref{Claim4} to finish the proof of \Cref{canonization}. \qedhere

\end{proof}

\begin{definition1}

Observe that the canonical function $\Gamma_{\gamma}$ we defined on $\mathcal{F}|Y$ is minimal in the following sense: If we take an arbitrary $Z \leq Y$ and define a parameter $\gamma'$ with respect to the separation types by $Z$, for any $a \in \mathcal{F}|Z$, we will have $\Gamma_{\gamma'}(a) = \Gamma_{\gamma}(a)$.

\end{definition1}

We can now obtain \Cref{lefcanon} as a corollary of \Cref{canonization}:

\begin{proof}[Proof of \Cref{lefcanon}]

Let $k \in \omega \setminus \{0\}$ and take a function $f : \mathrm{FIN}^{[k]} \to \omega$. Since $\mathrm{FIN}^{[k]}$ is a front, by \Cref{canonization}, we may find $Y \in \mathrm{FIN}^{[\infty]}$ and $\gamma : \bigcup_{i < k} [Y]^{[i]} \to \{\operatorname{sm}, \operatorname{min-sep}, \operatorname{max-sep}, \operatorname{minmax-sep}, \operatorname{sss}, \operatorname{vss}\}$ such that for all $a,a' \in [Y]^{[k]}$, $f(a) = f(a')$ if and only if $\Gamma_{\gamma}(a) = \Gamma_{\gamma}(a')$. For $a \in [Y]^{[k]}$, define $\mathcal{J}_a = (\gamma(\varnothing), \gamma(a \upharpoonright 1), \ldots, \gamma(a \upharpoonright (k-1)))$. Note that there are only finitely many $\mathcal{J}_a$'s. It follows by \Cref{nashwilliams} that there is a $\mathcal{J}$ and $Y' \leq Y$ such that $\mathcal{J}_a = \mathcal{J}$ for all $a \in [Y']^{[k]}$, and the result follows. \qedhere

\end{proof}

Before we conclude this section, recall that we fixed a stable ordered-union $\mathcal{U}$. The following lemma, which follows the proof of \Cref{4.15}, states that the canonical $Y \leq X$ can be chosen with $[Y] \in \mathcal{U}$ as long as we have $[X] \in \mathcal{U}$ to begin with, and will be used in the next section:

\begin{lemma}\label{remark}

Assume that $\mathcal{F}$ is a front on some $X$ with $[X] \in \mathcal{U}$, and $g : \mathcal{F} \to \omega$ is a function. Then there is a canonical $Y \leq X$ with $[Y] \in \mathcal{U}$. In particular, there is some $Y\leq X$ with $[Y] \in \mathcal{U}$ which satisfies the conclusion of \Cref{canonization}, as well as of \Cref{Claim1}, \Cref{Claim2}, \Cref{Claim3}, and \Cref{Claim4}, which means that $g$ is canonized on $\mathcal{F}|Y$.

\end{lemma}

\begin{proof}

We sketch the argument. One can prove by induction on $\alpha < \omega_1$ that the set of $Z \in \mathrm{FIN}^{[\infty]}$ for which $\mathcal{F}|Z$ is $\alpha$-uniform is a closed set. Moreover, for fixed $a,b \in \mathrm{FIN}^{[<\infty]}$, the set of $Z \in \mathrm{FIN}^{[\infty]}$ which decides $a$ and $b$ is coanalytic. Thus, if we let $\mathcal{A} = \{Z \in \mathrm{FIN}^{[\infty]} : \mathcal{F}|Z \ \text{is uniform and $Z$ decides $a$ and $b$}\}$, by \Cref{prop} part \ref{infty}, there will be some $[Z'] \in \mathcal{U}$ such that $[Z']^{[\infty]} \subseteq \mathcal{A}$ or $[Z']^{[\infty]} \cap \mathcal{A} = \varnothing$. The latter cannot happen by density of the block sequences that decide $a$ and $b$. It follows that, for any $b \in \mathrm{FIN}^{[<\infty]}$, one can find $[Y_0] \in \mathcal{U}$, $Y_0 \leq X$ which decides every $a$ with $\operatorname{max}(a(|a|-1)) \leq \operatorname{max}(b(|b|-1))$. Utilizing \Cref{prop} part \ref{sel} in place of \Cref{fusion}, one can find $[Y_1] \in \mathcal{U}$, $Y_1 \leq X$, which satisfies \ref{canon1} of canonical. Similarly, one can use \Cref{prop} part \ref{tayp} for \ref{canon2} of canonical, part \ref{Ramsey} for \ref{canon3} of canonical, and part \ref{NWp} for \ref{canon4} and \ref{canon5} of canonical to get $[Y] \in \mathcal{U}$ with $Y \leq X$, which is canonical for $g$. \qedhere

\end{proof}

\section{Initial Tukey Structure Below \texorpdfstring{$\mathcal{U}$}{Lg}}\label{5}

This section contains the main results of the paper, \Cref{main1} and \Cref{main2}. We fix a stable ordered-union ultrafilter $\mathcal{U}$ on $\mathrm{FIN}$. Since $\operatorname{min} : \mathrm{FIN} \to \omega$ and $\operatorname{min-sep} : \mathrm{FIN} \to \{\{n\} : n \in \omega\}$, the maps $\operatorname{min-sep}$ and $\operatorname{min}$ are technically different, as well as the other ones. The following is routine to show:

\begin{theorem2}

Assume that the maps $\operatorname{min-sep}$, $\operatorname{max-sep}$, $\operatorname{minmax-sep}$, and $\operatorname{sss}$ have disjoint images on $[X] \in \mathcal{U}$. Then the $\mathrm{RK}$-image $\operatorname{minmax-sep}(\mathcal{U} \upharpoonright [X])$ is an ultrafilter on $[X]^{[2]}$. Moreover, $\operatorname{min-sep}(\mathcal{U} \upharpoonright [X]) \cong \mathcal{U}_{\operatorname{min}}$, $\operatorname{max-sep}(\mathcal{U} \upharpoonright [X]) \cong \mathcal{U}_{\operatorname{max}}$ and $\operatorname{minmax-sep}(\mathcal{U} \upharpoonright [X]) \cong \mathcal{U}_{\operatorname{minmax}}$.

\end{theorem2}

Therefore, from now on, it is justified to use the maps $(\cdot)$ and $(\cdot)\operatorname{-sep}$ interchangeably.

For notational simplicity, define $\mathcal{B} = \{X \in \mathrm{FIN}^{[\infty]} : [X] \in \mathcal{U}\}$ and define $\mathcal{B}^{[<\infty]} = \{a \in \mathrm{FIN}^{[<\infty]} : a \sqsubset X \ \text{for some $X \in \mathcal{B}$}\}$. Also, for $X \in \mathcal{B}$, denote $\mathcal{B} \upharpoonright X = \{X' \in \mathcal{B} : X' \leq X\}$ and $\mathcal{B}^{[<\infty]} \upharpoonright X = \{a \in \mathrm{FIN}^{[<\infty]} : a \sqsubset X' \ \text{for some $X' \in \mathcal{B} \upharpoonright X$}\}$. We will need the following lemma:

\begin{lemma}[Theorem $56$, \cite{DMT}] \label{5.2}

Whenever $\mathcal{V}$ is a nonprincipal ultrafilter on $\omega$ and $f : \mathcal{B} \to \mathcal{V}$ is monotone (meaning that $X \leq Y \Rightarrow f(X) \subseteq f(Y)$) and cofinal, there is $X \in \mathcal{B}$ and a monotone function $\tilde{f} : \mathrm{FIN}^{[\infty]} \to \mathcal{P}(\omega)$ such that

\begin{enumerate}[label=(\roman*)]
    \item $\tilde{f}$ is continuous with respect to the metric topology on $\mathrm{FIN}^{[\infty]}$,
    \item $\tilde{f} \upharpoonright (\mathcal{B} \upharpoonright X) = f \upharpoonright (\mathcal{B} \upharpoonright X)$,
    \item There is $\hat{f} : \mathrm{FIN}^{[<\infty]} \to [\omega]^{<\omega}$ such that; $a \sqsubseteq b \in \mathrm{FIN}^{[<\infty]} \Rightarrow \hat{f}(a) \sqsubseteq \hat{f}(b)$, $a \leq b \in \mathrm{FIN}^{[<\infty]} \Rightarrow \hat{f}(a) \subseteq \hat{f}(b)$, and for all $Y \in \mathrm{FIN}^{[\infty]}$, $\tilde{f}(Y) = \bigcup_{k \in \omega} \hat{f}(r_k(Y))$.
\end{enumerate}

\end{lemma}

Define the following class of ultrafilters on $\omega$ for $\alpha < \omega_1$:

\begin{enumerate}[label=(\roman*)]
    \item $\mathcal{C}_0 = \{\mathcal{U}, \mathcal{U}_{\operatorname{min}}, \mathcal{U}_{\operatorname{max}}, \mathcal{U}_{\operatorname{minmax}}\}$,
    \item $\mathcal{C}_{\alpha+1} = \{\lim_{x \to \mathcal{V}} \mathcal{W}_x : \mathcal{V} \in \mathcal{C}_0, \text{and each} \ \mathcal{W}_x \in \mathcal{C}_{\alpha}\}$.
    \item $\mathcal{C}_{\gamma} = \bigcup_{\alpha < \gamma} \mathcal{C}_{\alpha}$ for limit $\gamma$,
    \item $C_{\omega_1} = \bigcup_{\alpha < \omega_1} \mathcal{C}_{\alpha}$.
\end{enumerate}

Before stating \Cref{main1}, let us note that given a front $\mathcal{F}$ on some $X \in \mathcal{B}$, we may construct a new ultrafilter on the base set $\mathcal{F}$ in the following way:

\begin{definition}\label{ufonfront}

Let $\mathcal{F}$ be a front on $X$. We define $\mathcal{U} \upharpoonright \mathcal{F} = \langle \{\mathcal{F}|Y : Y \leq X, [Y] \in \mathcal{U}\} \rangle$.

\end{definition}

\noindent By \Cref{prop} part \ref{NWp}, $\mathcal{U} \upharpoonright \mathcal{F}$ is an ultrafilter on the base set $\mathcal{F}$.

We will build on the technique Todorcevic developed in \cite{diltod} to prove that Ramsey ultrafilters are Tukey minimal, which was outlined in \cite{dense}, and was extended to other ultrafilters in \cite{r1}, \cite{r2}, \cite{dobell}, and \cite{DMT}.

\begin{theorem} \label{main1}

Let $\mathcal{V}$ be a nonprincipal ultrafilter on a countable index set $I$ with $\mathcal{U} \geq_T \mathcal{V}$. Then $\mathcal{V}$ is isomorphic to an ultrafilter from the class $\mathcal{C}_{\omega_1}$.

\end{theorem}

\begin{proof}

Let $\mathcal{V}$ be a nonprincipal ultrafilter, without loss of generality on $\omega$, such that $\mathcal{U} \geq_T \mathcal{V}$. By \Cref{2.3}, there is a monotone and cofinal $f : \mathcal{U} \to \mathcal{V}$. Define $f \upharpoonright \mathcal{B} : \mathcal{B} \to \mathcal{V}$ by $(f \upharpoonright \mathcal{B})(X) = f'([X])$. By \Cref{5.2}, find $X \in \mathcal{B}$ and $\hat{f} : {\mathrm{FIN}}^{[<\infty]} \to {[\omega]}^{< \omega}$ such that $a \sqsubseteq b \Rightarrow \hat{f}(a) \sqsubseteq \hat{f}(b)$, $a \subseteq b \Rightarrow \hat{f}(a) \subseteq \hat{f}(b)$, and for all $X' \in \mathcal{B} \upharpoonright X$, $(f \upharpoonright \mathcal{B})(X') = \bigcup_{n \in \omega} \hat{f}(r_n(X'))$. By \Cref{prop} part \ref{infty}, we can find $X' \in \mathcal{B} \upharpoonright X$ such that for all $Y \leq X'$, there is $a \in [Y]^{[<\infty]}$ with $\hat{f}(a) \neq \varnothing$.

Define $\mathcal{F} = \{a \in {\mathcal{B}}^{[<\infty]} \upharpoonright X' : \hat{f}(a) \neq \varnothing \ \text{and} \ \hat{f}(a \upharpoonright (|a| - 1)) = \varnothing\}$. By the minimality condition and the choice of $X'$, $\mathcal{F}$ is a front on $X'$. Without loss of generality we may assume that $\mathcal{F}$ is a uniform front on $X'$. Consider the map $g : \mathcal{F} | X' \to \omega$ given by $g(a) = \operatorname{min}(\hat{f}(a))$ for all $a \in \mathcal{F}$. The following is straightforward to show:

\begin{theorem1}
$g(\mathcal{U} \upharpoonright (\mathcal{F} | X')) \cong g(\mathcal{U} \upharpoonright \mathcal{F}) = \mathcal{V}$.
\end{theorem1}

By \Cref{remark}, we may take canonical $Y \in \mathcal{B} \upharpoonright X'$ which witnesses the conclusion of \Cref{canonization}, along with all of the results in \Cref{3}. Let $\gamma : (\hat{\mathcal{F}} \setminus \mathcal{F})|Y \to \{\operatorname{sm}, \operatorname{min-sep}, \operatorname{max-sep}, \operatorname{minmax-sep}, \operatorname{sss}, \operatorname{vss}\}$ be the corresponding parameter function (\Cref{parameter}).

Note that $\mathcal{U} \upharpoonright (\mathcal{F} | Y) \cong \mathcal{U} \upharpoonright (\mathcal{F} | X')$, hence $g(\mathcal{U} \upharpoonright (\mathcal{F} | Y)) \cong \mathcal{V}$. The rest of the proof will contain similar arguments to the related results in \cite{r1}, in \cite{r2}, or in \cite{DMT}. Define $$\mathcal{S} = \{\Gamma_{\gamma}(a) : a \in \mathcal{F}|Y\}.$$ Let $\mathcal{W}$ denote the $\mathrm{RK}$-image $\Gamma_{\gamma}(\mathcal{U} \upharpoonright (\mathcal{F}|Y))$, an ultrafilter on the base set $\mathcal{S}$.

\begin{theorem1}

$\mathcal{W} \cong g(\mathcal{U} \upharpoonright (\mathcal{F} | Y))$.

\end{theorem1}

\begin{proof}

Define $\varphi : \mathcal{S} \to \omega$ by $\varphi(\Gamma_{\gamma}(a)) = g(a)$, for each $a \in \mathcal{F} | Y$. Since $\Gamma_{\gamma}(a) = \Gamma_{\gamma}(a')$ if and only if $g(a) = g(a')$ for all $a,a' \in \mathcal{F} | Y$, $\varphi$ is well-defined and injective. Thus it suffices to show that $\varphi(\mathcal{W}) = g(\mathcal{U} \upharpoonright (\mathcal{F}|Y))$.

Indeed, let $Y' \in \mathcal{B} \upharpoonright Y$ be arbitrary. Then $\varphi(\Gamma_{\gamma}[\mathcal{F}|Y']) = g[\mathcal{F}|Y']$. Note that by definition of $\mathrm{RK}$-image (\Cref{rk}), $\{\Gamma_{\gamma}[\mathcal{F}|Y'] : Y' \in \mathcal{B} \upharpoonright Y\}$ is cofinal in $\Gamma_{\gamma}(\mathcal{U} \upharpoonright (\mathcal{F}|Y)) = \mathcal{W}$. Therefore, $\varphi$ sends a cofinal subset of $\mathcal{W}$ to a cofinal subset of $g(\mathcal{U} \upharpoonright (\mathcal{F}|Y))$, and we are done. \qedhere

\end{proof}

Let us assume that $\Gamma_{\gamma}$ is not constant, i.e., $\gamma$ is not the constant $\operatorname{sm}$-map; otherwise $\mathcal{W}$ is a principal ultrafilter. Now we classify $\mathcal{W}$. Let $$\hat{\mathcal{S}} = \{\Gamma_{\gamma}(a \upharpoonright k) : a \in \mathcal{F}|Y \land (k=|a| \lor |\Gamma_{\gamma}(a \upharpoonright k)| < |\Gamma_{\gamma}(a \upharpoonright (k+1))|)\}.$$

Since $\mathcal{F}|Y$ is a front on $Y$, $\hat{\mathcal{F}}|Y$ has no infinite ascending sequences with respect to $\sqsubset$. The following is a straightforward consequence of this fact and the construction of $\Gamma_{\gamma}$:

\begin{lemma} \label{5.6}

There is no $(\sigma_n)_{n \in \omega} \subseteq \hat{\mathcal{S}}$ such that $\sigma_i \sqsubset \sigma_{i+1}$ for all $i \in \omega$. 

\end{lemma}

For $\sigma \in \hat{\mathcal{S}} \setminus \mathcal{S}$, define the set $\mathcal{I}_{\sigma} = \{a \upharpoonright k :a \in \mathcal{F}|Y, \ \text{$k$ is maximal with}$ $\Gamma_{\gamma}(a \upharpoonright k) = \sigma\}$. Since $\sigma \notin \mathcal{S}$, by \Cref{Claim3}, any $a \upharpoonright k \in \mathcal{I}_{\sigma}$ is separated in some sense by $Y$.

For $\sigma \in \hat{\mathcal{S}} \setminus \mathcal{S}$, fix some $a_{\sigma} \upharpoonright k_{\sigma} \in \mathcal{I}_{\sigma}$. Set $l_{\sigma} = |\Gamma_{\gamma}(a_{\sigma} \upharpoonright k_{\sigma})|$. We define the following filter: $$\mathcal{W}_{\sigma} = \langle \{\Gamma_{\gamma}((a_{\sigma} \upharpoonright k_{\sigma})^{\smallfrown} x)(l_{\sigma}) : x \in \mathcal{F}_{(a_{\sigma} \upharpoonright k_{\sigma})}|Z\} : Z \in \mathcal{B} \upharpoonright (Y/(a_{\sigma} \upharpoonright k_{\sigma})) \rangle.$$ 

\begin{lemma}

For all $\sigma \in \hat{\mathcal{S}} \setminus \mathcal{S}$, $\mathcal{W}_{\sigma}$ is an ultrafilter isomorphic to an ultrafilter from the set $\{\mathcal{U}, \mathcal{U}_{\operatorname{min}}, \mathcal{U}_{\operatorname{max}}, \mathcal{U}_{\operatorname{minmax}}\}$.

\end{lemma}

\begin{proof}

Let $\sigma \in \hat{\mathcal{S}} \setminus \mathcal{S}$.

First, assume that $a_{\sigma} \upharpoonright k_{\sigma}$ is max-separated by $Y$. For all $Z \in \mathcal{B} \upharpoonright (Y/(a_{\sigma} \upharpoonright k_{\sigma}))$, it follows that $\{\Gamma_{\gamma}((a_{\sigma} \upharpoonright k_{\sigma})^{\smallfrown} x)(l_{\sigma}) : x \in \mathcal{F}_{(a_{\sigma} \upharpoonright k_{\sigma})}|Z\} = \operatorname{max-sep}''[Z]$, and the result follows.

Assume that $a_{\sigma} \upharpoonright k_{\sigma}$ is minmax-separated by $Y$. For all $Z \in \mathcal{B} \upharpoonright (Y/(a_{\sigma} \upharpoonright k_{\sigma}))$, it follows that $\{\Gamma_{\gamma}((a_{\sigma} \upharpoonright k_{\sigma})^{\smallfrown} x)(l_{\sigma}) : x \in \mathcal{F}_{(a_{\sigma} \upharpoonright k_{\sigma})}|Z\} = \operatorname{minmax-sep}''[Z]$, and the result follows.

Assume that $a_{\sigma} \upharpoonright k_{\sigma}$ is min-separated by $Y$ and suppose it is not the case that the least $k_{\sigma} < k < |a_{\sigma}|$ such that $a_{\sigma} \upharpoonright k$ is separated in some sense by $Y$ is max-separated by $Y$. For all $Z \in \mathcal{B} \upharpoonright (Y/(a_{\sigma} \upharpoonright k_{\sigma}))$, it follows that $\{\Gamma_{\gamma}((a_{\sigma} \upharpoonright k_{\sigma})^{\smallfrown} x)(l_{\sigma}) : x \in \mathcal{F}_{(a_{\sigma} \upharpoonright k_{\sigma})}|Z\} = \operatorname{min-sep}''[Z]$, and the result follows.

Assume that $a_{\sigma} \upharpoonright k_{\sigma}$ is min-separated by $Y$, and suppose there is least $k_{\sigma} < k < |a_{\sigma}|$ such that $a_{\sigma} \upharpoonright k$ is separated in some sense by $Y$, and moreover $a_{\sigma} \upharpoonright k$ is max-separated by $Y$. For all $Z \in \mathcal{B} \upharpoonright (Y/(a_{\sigma} \upharpoonright k_{\sigma}))$, it follows that $\{\Gamma_{\gamma}((a_{\sigma} \upharpoonright k_{\sigma})^{\smallfrown} x)(l_{\sigma}) : x \in \mathcal{F}_{(a_{\sigma} \upharpoonright k_{\sigma})}|Z\} \subseteq \operatorname{minmax-sep}''[Z]$. Conversely, let $Z \in \mathcal{B} \upharpoonright (Y/(a_{\sigma} \upharpoonright k_{\sigma}))$. Define $c : [Z] \to 2$ by $c(s) = 0$ if and only if there is $x \in \mathcal{F}_{(a_{\sigma} \upharpoonright k_{\sigma})} | Z$ such that $\Gamma_{\gamma}((a_{\sigma} \upharpoonright k_{\sigma})^{\smallfrown} x)(l_{\sigma}) = \operatorname{minmax-sep}(s)$, for all $s \in [Z]$. By \Cref{prop} part \ref{Ramsey}, find $Z' \in \mathcal{B} \upharpoonright Z$ such that $c \upharpoonright [Z']$ is constant. Take any $y \in \mathcal{F}_{(a_{\sigma} \upharpoonright k_{\sigma})} | Z'$. Let $a' = (a_{\sigma} \upharpoonright k_{\sigma})^{\smallfrown} y$. Find the least $k_{\sigma} < j < |a'|$ such that $a' \upharpoonright j$ is separated in some sense by $Y$. It follows that $\gamma(a' \upharpoonright j)=\operatorname{max-sep}$. But then $\Gamma_{\gamma}((a_{\sigma} \upharpoonright k_{\sigma})^{\smallfrown} y)(l_{\sigma}) = \operatorname{minmax-sep}(y(0) \cup y(j-k_{\sigma}))$. It follows that $c(y(0) \cup y(j-k_{\sigma})) = 0$, and so $c[Z'] = \{0\}$. But then $\operatorname{minmax-sep}''[Z'] \subseteq \{\Gamma_{\gamma}((a_{\sigma} \upharpoonright k_{\sigma})^{\smallfrown} x)(l_{\sigma}) : x \in \mathcal{F}_{(a_{\sigma} \upharpoonright k_{\sigma})}|Z\}$. Therefore, $\mathcal{W}_{\sigma}$ is cofinal in $\langle \operatorname{minmax-sep}''[Z] : Z \in \mathcal{B} \upharpoonright (Y/(a_{\sigma} \upharpoonright k_{\sigma})) \rangle$ and $\langle \operatorname{minmax-sep}''[Z] : Z \in \mathcal{B} \upharpoonright (Y/(a_{\sigma} \upharpoonright k_{\sigma})) \rangle$ is cofinal in $\mathcal{W}_{\sigma}$, and the result follows.

Assume that $a_{\sigma} \upharpoonright k_{\sigma}$ is strongly separated by $Y$. For all $Z \in \mathcal{B} \upharpoonright (Y/(a_{\sigma} \upharpoonright k_{\sigma}))$, it follows that $\{\Gamma_{\gamma}((a_{\sigma} \upharpoonright k_{\sigma})^{\smallfrown} x)(l_{\sigma}) : x \in \mathcal{F}_{(a_{\sigma} \upharpoonright k_{\sigma})}|Z\} \subseteq [Z]$. Conversely, let $Z \in \mathcal{B} \upharpoonright (Y/(a_{\sigma} \upharpoonright k_{\sigma}))$. Define $c : [Z] \to 2$ by $c(s) = 0$ if and only if there is $x \in \mathcal{F}_{(a_{\sigma} \upharpoonright k_{\sigma})} | Z$ such that $\Gamma_{\gamma}((a_{\sigma} \upharpoonright k_{\sigma})^{\smallfrown} x)(l_{\sigma}) = s$, for all $s \in [Z]$. Take any $y \in \mathcal{F}_{(a_{\sigma} \upharpoonright k_{\sigma})} | Z'$. Let $a' = (a_{\sigma} \upharpoonright k_{\sigma})^{\smallfrown} y$. Let $k_{\sigma} < j_0 < \ldots < j_k < k_{\sigma} + |y|$ increasingly enumerate those $j$'s such that $a' \upharpoonright j$ is strongly separated by $Y$. It follows that $\gamma(a' \upharpoonright j_i) = \operatorname{sss}$ for all $i < k$ and $\gamma(a' \upharpoonright j_k) = \operatorname{vss}$. But then $\Gamma_{\gamma}((a_{\sigma} \upharpoonright k_{\sigma})^{\smallfrown} y)(l_{\sigma}) = \bigcup_{k_{\sigma} \leq i \leq j_k} y(i-k_{\sigma})$. Thus, $c(\bigcup_{k_{\sigma} \leq i \leq j_k} y(i-k_{\sigma})) = 0$ and so $c[Z'] = \{0\}$. We see that $[Z'] \subseteq \{\Gamma_{\gamma}((a_{\sigma} \upharpoonright k_{\sigma})^{\smallfrown} x)(l_{\sigma}) : x \in \mathcal{F}_{(a_{\sigma} \upharpoonright k_{\sigma})}|Z\}$, and the result follows similarly to above. \qedhere

\end{proof}

Recall the following notion used in \cite{r1}, \cite{r2}, and \cite{DMT}:

\begin{definition}

We call $T \subseteq \hat{\mathcal{S}}$ a \textit{tree} if $\varnothing \in T$, $T$ is closed under initial segments, and maximal nodes of $T$ are in $\mathcal{S}$. We call $T \subseteq \hat{\mathcal{S}}$ \emph{well-founded} if $T$ does not have any infinite cofinal branches. We call $T$ a \textit{$\vec{\mathcal{W}}$-tree} if $T$ is a tree and for all $\Gamma_{\gamma}(a \upharpoonright k) \in T \cap (\hat{\mathcal{S}} \setminus \mathcal{S})$, where $a \in \mathcal{F}|Y$ and $|\Gamma_{\gamma}(a \upharpoonright k)| < |\Gamma_{\gamma}(a \upharpoonright (k+1))|$, if we denote $l=|\Gamma_{\gamma}(a \upharpoonright k)|$, then $\{\Gamma_{\gamma}(b \upharpoonright j)(l) : \Gamma_{\gamma}(a \upharpoonright k) \sqsubset \Gamma_{\gamma}(b \upharpoonright j) \in T, \ \text{for some $b \in \mathcal{F}|Y$ and $j \leq |b|$}\} \in \mathcal{W}_a$. For a well-founded $\vec{\mathcal{W}}$-tree $T$, $[T] = T \cap \mathcal{S}$ will denote the set of cofinal branches through $T$.

\end{definition}

By \Cref{5.6}, $\hat{\mathcal{S}} \setminus \mathcal{S}$ is a well-founded tree. It is straightforward to check that the set of $[T]$'s, where $T \subseteq \hat{\mathcal{S}}$ is a $\vec{\mathcal{W}}$-tree, generates a filter on $\mathcal{S}$. Let us call this filter $\mathcal{T}$. Following is the key result:

\begin{lemma}

$\mathcal{V} \cong \mathcal{W} = \Gamma_{\gamma}(\mathcal{U} \upharpoonright (\mathcal{F}|Y)) = \mathcal{T}$. Therefore, $\mathcal{V}$ is isomorphic to an ultrafilter of $\vec{\mathcal{W}}$-trees, where $\hat{\mathcal{S}} \setminus \mathcal{S}$ is a well-founded tree, $\vec{\mathcal{W}}=(\mathcal{W}_{\sigma} : \sigma \in \hat{\mathcal{S}} \setminus \mathcal{S})$, and each $\mathcal{W}_{\sigma}$ is isomorphic to exactly one of $\mathcal{U}$, $\mathcal{U}_{\operatorname{min}}$, $\mathcal{U}_{\operatorname{max}}$, or $\mathcal{U}_{\operatorname{minmax}}$.

\end{lemma}

\begin{proof}

It suffices to show that $\mathcal{T}$ contains a cofinal subset of $\mathcal{W}$. Let $Z \in \mathcal{B} \upharpoonright Y$ be arbitrary. Set $S = \{\Gamma_{\gamma}(a) : a \in \mathcal{F}|Z\}$. Let $\hat{S}$ denote the set of initial segments of elements of $S$ (including the empty sequence). The proof will be complete if we can show that $\hat{S}$ is a $\vec{\mathcal{W}}$-tree, since $\mathcal{W} = \langle \{\Gamma_{\gamma}(a) : a \in \mathcal{F}|Z\} : Z \in \mathcal{B} \upharpoonright Y \rangle$ and the set of cofinal branches through $\hat{S}$ is $S$.

Indeed, let $\sigma \in \hat{S} \setminus S$. Pick $a \in \mathcal{F}|Z$ and maximal $k < |a|$ such that $\Gamma_{\gamma}(a \upharpoonright k) = \sigma$. It follows that $a \upharpoonright k$ and $a_{\sigma} \upharpoonright k_{\sigma}$ are mixed by $Y$, and hence by $Z$. Set $l = |\Gamma_{\gamma}(a \upharpoonright k)|$. Let $\mathcal{X} = \{x \in (\mathcal{F}_{(a_{\sigma} \upharpoonright k_{\sigma})})|Z) : (\exists y \in (\mathcal{F}_{(a \upharpoonright k)})|Z) \ \Gamma_{\gamma}((a_{\sigma} \upharpoonright k_{\sigma})^{\smallfrown}x)(l_{\sigma}) = \Gamma_{\gamma}((a\upharpoonright k)^{\smallfrown}y)(l)\}.$ 

By \Cref{prop} part \ref{NWp}, we can find $Z' \in \mathcal{B} \upharpoonright Z$ such that $\mathcal{F}_{(a_{\sigma} \upharpoonright k_{\sigma})}|Z' \subseteq \mathcal{X}$ or $(\mathcal{F}_{(a_{\sigma} \upharpoonright k_{\sigma})}|Z') \cap \mathcal{X} = \varnothing$. The latter case cannot happen since $a \upharpoonright k$ and $a_{\sigma} \upharpoonright k_{\sigma}$ are mixed. Therefore, $$\{\Gamma_{\gamma}(b \upharpoonright j)(l) : \Gamma_{\gamma}(a \upharpoonright k) \sqsubset \Gamma_{\gamma}(b \upharpoonright j) \in T, \ \text{for some $b \in \mathcal{F}|Y$ and $j \leq |b|$}\}$$ $$\supseteq \{\Gamma_{\gamma}((a \upharpoonright k)^{\smallfrown}x)(l) : x \in \mathcal{F}_{(a \upharpoonright k)} | Z\}$$ $$\supseteq \{\Gamma_{\gamma}((a_{\sigma} \upharpoonright k_{\sigma})^{\smallfrown}x)(l_{\sigma}) : x \in \mathcal{F}_{(a_{\sigma} \upharpoonright k_{\sigma})} | Z'\}.$$ Since $\{\Gamma_{\gamma}((a_{\sigma} \upharpoonright k_{\sigma})^{\smallfrown}x)(l_{\sigma}) : x \in \mathcal{F}_{(a_{\sigma} \upharpoonright k_{\sigma})} | Z'\} \in \mathcal{W}_{\sigma}$, we are done. \qedhere

\end{proof}

For $\sigma, \tau \in \hat{\mathcal{S}}$, we define $\sigma \preccurlyeq \tau$ if and only if $\tau \sqsubseteq \sigma$. By \Cref{5.6}, $\preccurlyeq$ is well-founded on $\hat{\mathcal{S}}$. Hence, for $\sigma \in \hat{\mathcal{S}}$, we can define $\operatorname{rk}(\sigma) = \{\operatorname{rk}(\tau) : \tau \prec \sigma\}$. We let $\operatorname{rk}(\mathcal{S})=\operatorname{rk}(\varnothing)$. The following now follows by induction on $\operatorname{rk}(\mathcal{S})$:

\begin{theorem1}

$\mathcal{T}$ is isomorphic to an ultrafilter from the set $\mathcal{C}_{\omega_1}$. In other words, $\mathcal{T}$ is a countable Fubini iterate of $\mathcal{U}$, $\mathcal{U}_{\operatorname{min}}$, $\mathcal{U}_{\operatorname{max}}$, and $\mathcal{U}_{\operatorname{minmax}}$.

\end{theorem1}

As $\mathcal{V} \cong \mathcal{W} = \mathcal{T}$, we see that $\mathcal{V}$ is a countable Fubini iterate of $\mathcal{U}$, $\mathcal{U}_{\operatorname{min}}$, $\mathcal{U}_{\operatorname{max}}$, and $\mathcal{U}_{\operatorname{minmax}}$. This completes the proof of \Cref{main1}. \qedhere

\end{proof}

\begin{theorem} \label{main2}

Let $\mathcal{V}$ be a nonprincipal ultrafilter on a countable index set $I$ with $\mathcal{U} \geq_T \mathcal{V}$. Then $\mathcal{V}$ is Tukey equivalent to an ultrafilter from the set $\{\mathcal{U}, \mathcal{U}_{\operatorname{min}}, \mathcal{U}_{\operatorname{max}},$ $\mathcal{U}_{\operatorname{minmax}}\}$.

\end{theorem}

\begin{proof}

By \Cref{main1}, $\mathcal{V}$ is isomorphic to an ultrafilter from the class $\mathcal{C}_{\omega_1}$. It suffices to show that the ultrafilters in $\mathcal{C}_{\alpha}$ are Tukey equivalent to one from the set $\{\mathcal{U}, \mathcal{U}_{\operatorname{min}}, \mathcal{U}_{\operatorname{max}}, \mathcal{U}_{\operatorname{minmax}}\} = \mathcal{C}_0$. This is indeed clear for $\alpha=0$. Now assume that every ultrafilter from $\mathcal{C}_{\beta}$ is Tukey equivalent to one from $\mathcal{C}_0$ for all $\beta < \alpha$. Take $\mathcal{W} \in \mathcal{C}_{\alpha}$. Write $\mathcal{W} = \lim_{i \to \mathcal{W}'} \mathcal{W}_i$, where $\mathcal{W}' \in \mathcal{C}_0$ and each $\mathcal{W}_i \in \mathcal{C}_{\beta}$ for some $\beta < \alpha$. By the induction assumption, each $\mathcal{W}_i$ is Tukey equivalent to an ultrafilter from $\mathcal{C}_0$. Let $\mathcal{V}' \in \mathcal{C}_0$ be the one for which there is some $W \in \mathcal{W}'$ with $\mathcal{W}_i \equiv_T \mathcal{V}'$ for every $i \in W$. It follows that $\mathcal{W} \equiv_T \mathcal{W}' \cdot \mathcal{V}'$, where $\mathcal{W}', \mathcal{V}' \in \mathcal{C}_0$. By \Cref{tukey}, we see that $\mathcal{W}' \cdot \mathcal{V}'$ is Tukey equivalent to an ultrafilter from $\mathcal{C}_0$. \qedhere

\end{proof}

By \Cref{main1}, \Cref{main2}, \Cref{tukey}, \Cref{3.2}, and \Cref{minimality}, we can now give the following description of each Tukey class below $\mathcal{U}$:

\begin{corollary}

$[\mathcal{U}_{\operatorname{min}}]$ and $[\mathcal{U}_{\operatorname{max}}]$ consist of countable Fubini iterates of $\mathcal{U}_{\operatorname{min}}$ and $\mathcal{U}_{\operatorname{max}}$, respectively. $[\mathcal{U}_{\operatorname{minmax}}]$ consists of countable Fubini iterates of $\mathcal{U}_{\operatorname{min}}$ and $\mathcal{U}_{\operatorname{max}}$ that include both of $\mathcal{U}_{\operatorname{min}}$ and $\mathcal{U}_{\operatorname{max}}$ at a step of the iteration. Finally, $[\mathcal{U}]$ consists of countable Fubini iterates of $\mathcal{U}$, $\mathcal{U}_{\operatorname{min}}$, and $\mathcal{U}_{\operatorname{max}}$ that include $\mathcal{U}$ at a step of the iteration.

\end{corollary}

\bibliographystyle{amsalpha}
\bibliography{refs}

\end{document}